\documentclass[12pt,oneside]{book}

\usepackage{tesis-epn}
\DeclareMathAlphabet{\mymathbb}{U}{BOONDOX-ds}{m}{n}

\usepackage{mathpazo}
\usepackage{enumitem}
\usepackage{amsmath,amssymb,amsthm}

\usepackage{verbatim}

\usepackage{pstricks-add}
\usepackage{float}
\usepackage{bbm}

\newtheoremstyle{estiloejemplo}%
    {9pt}{9pt}%
    {}%
    {0pt}%
    {\bfseries\scshape}{.}%
    { }%
    {}

\theoremstyle{estiloejemplo}
    \newtheorem{ejemplo}{Example}
    \newtheorem*{observacion}{Observation}

\newtheoremstyle{estiloteorema}%
    {9pt}{9pt}%
    {\slshape}%
    {0pt}%
    {\bfseries\scshape}{.}%
    { }%
    {}

\theoremstyle{estiloteorema}
    \newtheorem{definicion}{Definition}[chapter]
    
    \newtheorem{teorema}{Theorem}[chapter]
    
    \newtheorem{proposicion}[teorema]{Proposition}
    \newtheorem{lema}[teorema]{Lemma}

\newcommand{\funcion}[5]{%
    {\setlength{\arraycolsep}{2pt}
        \begin{array}{r@{}ccl}
        #1\colon & #2 & \longrightarrow & #3\\
        & #4 & \longmapsto & #5
        \end{array}}
    }

\newcommand{\func}[3]{%
    #1\colon  #2  \to  #3
}

\newcommand{\R}{\mathbb{R}}

\setlist[itemize,1]{label=$\bullet$}

\allowdisplaybreaks 

\titulo{Category Theory with examples in probability theory}
\autor{William Gabriel Granda Betancourt}
\correoautor{\url{william.granda@epn.edu.ec}}
\director{Ph. D. Miguel Alfonso Flores Sánchez}
\correodirector{\url{miguel.flores@epn.edu.ec}}
\fecha{September 2021}

\begin{document}

\frontmatter
\portadilla


\tableofcontents


\mainmatter
\addtolength{\abovedisplayskip}{-1mm}
\addtolength{\belowdisplayskip}{-1mm}

\chapter*{Abstract}
\addcontentsline{toc}{chapter}{Abstract}

The basic concepts of category theory are developed and examples of them are presented to illustrate them using measurement theory and probability theory tools. Motivated by Perrone's work~\cite{PerroneP} where notes on category theory are developed with examples of basic mathematics, we present the concepts of category, functor, natural transformation, and products with examples in the probabilistic context.

The most prominent examples of the application of Category Theory to Probability Theory are the Lawvere ~\cite{LawvereF} and Giry~\cite{GiryM} approaches. However, there are few categories with objects as probability spaces due to the difficulty of finding an appropriate condition to define arrows between them.

\include{Capitulos/01Introduccion}
\include{Capitulos/02Probabilidad}
\chapter{Category Theory}\label{cap:capitulo3}

\section{Categories}

\begin{definicion}[Category]\label{def:categoria}
    A \emph{category} $ \mathcal{C}$ consists of:
    \begin{itemize}
        \item 
           A collection of objects, denoted by $ \textbf{Ob}(\mathcal{C})$;
        \item
            A collection of arrows, denoted by $ \boldsymbol{\mathcal{A}} (\mathcal{C}) $;
        \item
           Each arrow is assigned two operations, called domain and codomain. The domain assigns each arrow $ f $ an object $ A = \text{dom} (f) $, and the codomain assigns each arrow $ f $ an object $ B = \text{cod} (f) $. We will use the notation
            \[
                \func{f}{A}{B}. 
            \]
        \item
            For each pair of arrows $ f, g \in \boldsymbol{\mathcal {A}} (C) $, such that $ \text{dom}(g) = \text{cod} (f) $, the operation composition of $ f $ and $ g $, denoted by $ g \circ f $, such that $ \func {g \circ f} {\text{dom} (f)}{\text{cod}(g)} $. If so, we say that $ f $ and $ g $ are composable arrows.
        \item
           Let $f,g,h\in\boldsymbol{\mathcal{A}}(C)$ arrows such that  $\text{dom}(g)=\text{cod}(f)$ and $\text{dom}(h)=\text{cod}(g)$, we have to 
            \[
                h\circ (g\circ f) = (h\circ g)\circ f,
            \]
           that is, the composition is associative.
        \item
           For each object  $A\in\textbf{Ob}(\mathcal{A})$, there is an arrow $\func{1_A}{A}{A}$,  which we call the identity arrow, such that if $\func{f}{A}{B}$ y $\func{g}{C}{A}$ are arrows, then it is satisfied that 
            \[
                f\circ 1_A = f \quad \text{and} \quad g\circ 1_A = g. 
            \]
    \end{itemize}
\end{definicion}

To illustrate the concept of category, we first present some classic examples of categories.

\begin{ejemplo}
    \begin{enumerate}
        \item 
            \textbf{Category one:} This category has an object and an arrow, which corresponds to the identity arrow of its only object.
            \begin{center}
                \begin{pspicture}[showgrid=false](0,0.2)(1,1)
                    \psdot(0.5,0.5)
                \end{pspicture}
            \end{center}
           We will denote it by \textbf{1}. 
        \item
           \textbf{Category two:} It is the category that has two objects, an identity arrow for each object, and a single arrow between the two objects. Its diagram is:
            \begin{center}
                \begin{pspicture}[showgrid=false](0.3,0.3)(1,1)
                    \psdots(0,0.5)(0.8,0.5)
                    \psline{->}(0.2,0.5)(0.6,0.5)
                \end{pspicture}
            \end{center}
            It is denoted by \textbf{2}.
        \item
            \textbf{Category three:} This category has three objects, three identity arrows for each object, exactly one arrow from first to second object, exactly one arrow from second to third object, and exactly one arrow from first to third object. Its diagram is given by:
            \begin{center}
                \psset{unit=0.8}
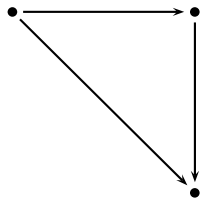
                \begin{pspicture}[showgrid=false](0,0)(3,3)
                    \pnode(3,3){A}
                    \pnode(0,3){B}
                    \pnode(3,0){C}
                    \psdots(3,3)(0,3)(3,0)
                    \ncline[nodesep=4pt]{->}{B}{A}
                    \ncline[nodesep=4pt]{->}{A}{C}
                    \ncline[nodesep=4pt]{->}{B}{C}
                \end{pspicture}
            \end{center}
            This category will be denoted by \textbf{3}. 
        \item
            \textbf{Sets category:} It is the category defined by:  
            \begin{itemize}
                \item 
                    \textbf{Objects}: Sets;
                \item
                    \textbf{Arrows}: Functions between sets; 
                \item
                    \textbf{Composition}: Usual composition of functions; 
                \item
                    \textbf{Identity:} Identity function defined on a set. 
            \end{itemize}
        \item
            \textbf{Preorders:} 
               Let $ X $ be a set and $ \leq $ a relation on $ X $, $ \leq $ is said to be a preorder if it is reflexive and transitive. The category \textbf{Pre} is defined by means of the following information: Its objects are the elements of $ X $, given two objects $ x, y \ in X $ there is a single arrow, if and only, if $ x \leq y $, the composition between two arrows is given by the transitivity and the identities for each object are given by the reflexivity.
    \end{enumerate}
\end{ejemplo}

Thanks to the fact that the composition of functions is measurable, we can define the category of measurable spaces.

\begin{ejemplo}
    \textbf {Category of measurable spaces:} The category \textbf{Meas} is defined whose objects are measurable spaces and whose arrows between measurable spaces are measurable functions.
\end{ejemplo}

Now, we introduce the category $ \chi $, defined in \cite{TakanoriA}.

\begin{definicion}[Category $\chi$]\label{def:categoriachi}
    Let $(X,\mathcal{F}_X)$ be a measurable space and $\chi:=\chi(X,\mathcal{F}_X)$ the set of all pairs of the form $(\mathcal{G},\mathbb{P}_X)$, where $\mathcal{G}\subset\mathcal{F}_X$ is a  sub$-\sigma-$algebra of $\mathcal{F}_X$ and $\mathbb{P}_X$ ess a probability measure defined on $(X,\mathcal{F}_X)$. For an element $\mathcal{U}\in\chi$, we denote its $\sigma-$áalgebra and its probability measure by  $\mathcal{F}_{\mathcal{U}}$ y $\mathbb{P}_{\mathcal{U}}$, respectively, that is, $\mathcal{U} = (\mathcal{F}_{\mathcal{U}},\mathbb{P}_{\mathcal{U}})$. For $\mathcal{U},\mathcal{V}\in\chi$, consider the binary relation, denoted by $\leq_{\chi}$, and defined as follows 
    \[  
        \mathcal{V}\leq_{\chi} \mathcal{U} \;\text{ si y solo si }\, \mathcal{F}_{\mathcal{V}} \subset \mathcal{F}_{\mathcal{U}} \;\text{ y }\; \mathbb{P}_{\mathcal{V}} \gg \mathbb{P}_{\mathcal{U}}.
    \]
    The category $ \chi $ is defined as one that has exactly one arrow $\func{f^{\mathcal{V}}_{\mathcal{U}}}{\mathcal{V}}{\mathcal{U}}$ in $\chi$ if and only if $\mathcal{V}\leq_{\chi}\mathcal{U}$. 
\end{definicion}

\begin{observacion}
        Note that $(\chi,\leq_{\chi})$ is a preordered set. Indeed, let $\mathcal{U}\in\chi$, the reflexivity is satisfied given that  $\mathcal{F}_{\mathcal{U}} \subseteq \mathcal{F}_{\mathcal{U}}$ and since $\mathbb{P}_{\mathcal{U}}$ is absolutely continuous with respect to itself. Thus, it remains to prove the transitivity, for this, let $\mathcal{U},\mathcal{V},\mathcal{W}\in\chi$, such that 
        \[
            \mathcal{V}\leq_{\chi} \mathcal{U} \quad\text{y}\quad\mathcal{U}\leq_{\chi} \mathcal{W}, 
        \]
       we are going to show that $\mathcal{V}\leq_{\chi} \mathcal{W}$. On the one hand, we have to 
        \[
            \mathcal{F}_{\mathcal{V}} \subset \mathcal{F}_{\mathcal{U}}\quad\text{y}\quad \mathcal{F}_{\mathcal{U}} \subset \mathcal{F}_{\mathcal{W}},  
        \]
        from where, $\mathcal{F}_{\mathcal{V}}\subset\mathcal{F}_{\mathcal{V}}$. 
        
        On the other hand, let  $A\in\mathcal{F}_{\mathcal{V}}$ such that $\mathbb{P}_{\mathcal{V}}(A) = 0$, we must prove that $\mathbb{P}_{\mathcal{W}}(A) = 0$. Since $\mathbb{P}_{\mathcal{U}} \ll \mathbb{P}_{\mathcal{V}}$, we have $\mathbb{P}_{\mathcal{U}}(A) = 0$, from where, using the fact that $\mathbb{P}_{\mathcal{W}} \ll \mathbb{P}_{\mathcal{U}}$, we obtain that $\mathbb{P}_{\mathcal{W}}(A) = 0$. 
        
        Therefore, the category $ \chi $ is well defined. We will return to this category later to present examples of other category theory concepts. 
\end{observacion}

To present the following example, it is necessary to consider the following definition:

\begin{definicion}[Bounded function (Motoyama-Tanaka~\cite{HitoshiM})]\label{def:flechasacotadas}
    Let $(X,\mathcal{F}_X,\mathbb{P}_X)$ and $(Y,\mathcal{F}_Y,\mathbb{P}_Y)$ be two probability spaces. Let  $\func{f}{(Y,\mathcal{F}_Y,\mathbb{P}_Y)}{(X,\mathcal{F}_X,\mathbb{P}_X)}$ a function $\mathcal{F}_Y/\mathcal{F}_X-$ easurable, it is said that $f$ is \emph{bounded} if there exists a positive number  $M>0$ such that 
    \[
        \mathbb{P}_Y(f^{-1}(A)) \leq M\mathbb{P}_X (A),  
    \]
    for all $A\in\mathcal{F}_X$. 
\end{definicion}

Thus, we have the following example of a category in the probabilistic context:

\begin{definicion}[Category \textbf{CPS}]\label{def:catCPS}
    The \textbf{CPS} \emph{category} is defined by the following information: 
    \begin{itemize}
        \item 
            \textbf{Objetos:} Probability spaces;
        \item
            \textbf{Flechas:} Let $\overline{X} = (X,\mathcal{F}_X,\mathbb{P}_X)$ y $\overline{Y}=(Y,\mathcal{F}_Y,\mathbb{P}_Y)$ two probability spaces, the set of arrows between these spaces is defined by
            \[
                \textbf{CPS}(\overline{X},\overline{Y}) = \{f|\func{f}{\overline{X}}{\overline{Y}} \; \text{is a bounded function.}\}.
            \]
        \item
           \textbf{Composition:} Usual composition of functions;
        \item
            \textbf{Identidades:} If $\func{Id}{(X,\mathcal{F}_X)}{(X,\mathcal{F}_X)}$ represents the identity function (the same as $\mathcal{F}_X-$measurable), then the identity arrow for a probability space in \textbf{CPS} is denoted by 
            \[
                \func{id}{(X,\mathcal{F}_X,\mathbb{P}_X)}{(X,\mathcal{F}_X,\mathbb{P}_X).}
            \]
            Note that $id$ is a bounded function. Indeed, for all $ A \in\mathcal{F}_X $, we have to
            \[
                \mathbb{P}_X \left(id^{-1}(A)\right) = \mathbb{P}_X(A). 
            \]
    \end{itemize}
\end{definicion}

\begin{observacion}
    \textbf{CPS} is a category.

     It is enough to show that the composition of bounded functions is bounded, for this, they are $\func{f}{(X,\mathcal{F}_X,\mathbb{P}_X)}{(Y,\mathcal{F}_Y,\mathbb{P}_Y)}$ and $\func{g}{(Y,\mathcal{F}_Y,\mathbb{P}_Y)}{(Z,\mathcal{F}_Z,\mathbb{P}_Z)}$ two arrows in \textbf{CPS}, we are going to show that $\func{g\circ f}{(X,\mathcal{F}_X,\mathbb{P}_X)}{(Z,\mathcal{F}_Z,\mathbb{P}_Z)}$ is bounded. Let $A\in\mathcal{F}_Z$, we must prove that there exists $ M> 0 $ such that 
    \[
            \mathbb{P}_X \left( f^{-1} \left(g^{-1} (A)\right)  \right) \leq M\mathbb{P}_Z (A).
    \]
   First of all, we know there exists $ M_1> 0 $ such that
    \begin{equation}\label{eqn:prop3.1_01}
        \mathbb{P}_Y (g^{-1}(A))\leq M_1 \mathbb{P}_Z(A).
    \end{equation}
    Now, exists $M_2>0$ such that
    \begin{equation}\label{eqn:prop3.1_02}
        \mathbb{P}_X\left( f^{-1}\left( g^{-1}(A) \right) \right) \leq M_2 \mathbb{P}_Y(g^{-1}(A),
    \end{equation}
   for $g^{-1}(A)\in\mathcal{F}_Y$.

    Then, combining  \eqref{eqn:prop3.1_01} and \eqref{eqn:prop3.1_01}, we get that 
    \[
        \mathbb{P}_X\left( f^{-1}\left( g^{-1}(A) \right) \right) \leq M_2\cdot M_1 \mathbb{P}_Z(A),
    \]
   from where, taking $M:=M_1\cdot M_2 >0$.

    Therefore, we have proven that $ g \circ f $ is an arrow in \textbf{CPS}.
\end{observacion}

\begin{observacion}\label{obs:catCMS}
    In fact, Hitoshi \cite{HitoshiM} defines the category \textbf{CMS}, whose objects are measured spaces and the arrows between two measured spaces $(X,\mathcal{F}_X,\mu)$ y $(Y,\mathcal{F}_Y,\nu)$ they are bounded functions.
\end{observacion}

Finally, we present the category of stochastic functions defined in \cite{LawvereF}, for this, it is necessary to consider the following definition:

   \begin{definicion}[Stochastic function]
    Let $(X,\mathcal{F}_X)$ and $(Y,\mathcal{F}_Y)$ measurable spaces. A \emph{stochastic function} $\func{T}{(X,\mathcal{F}_X)}{(Y,\mathcal{F}_Y)}$ is a function
    \[
        \func{T}{X\times \mathcal{F}_Y}{[0,1]}, 
    \]
    such that  
    \begin{enumerate}
        \item 
             For all $x\in X$, the function $\func{T_x:=T(x,\cdot)}{\mathcal{F}_Y}{[0,1]}$ is a probability measure in $(Y,\mathcal{F}_Y)$, that is, $T_x$ satisfies the following properties: 
               \begin{enumerate}
                    \item 
                        $0\leq T_x (F_Y) \leq 1$ for all $F_Y\in\mathcal{F}_Y$; 
                    \item
                        $ T_x(Y) = 1$;
                    \item
                        $T_x \left( \displaystyle\bigcup_{n\in\mathbb{N}} F_n \right) = \displaystyle\sum_{n\in\mathbb{N}} T_x (F_n)$ for any $\{F_n\}_{n\in\mathbb{N}}\subset\mathcal{F}_Y$ collection of two-by-two disjoint events. 
                \end{enumerate}
            \item
               For each $F_Y\in\mathcal{F}_Y$, the function $\func{T_{F_Y}:= T(\cdot,F_Y)}{X}{[0,1]}$ is $\mathcal{F}_X/\mathcal{B}([0,1])-$ measurable. 
    \end{enumerate}
    
     We will call $T(x,F_Y)$ the $ T $ -probability (conditional) of the event $F_Y$ on $(Y,\mathcal{F}_Y)$ given the point $ x $ at  $(X,\mathcal{F}_X)$.
\end{definicion}

Let's present some examples to illustrate the concept of stochastic function:

\begin{ejemplo}[Characteristic function]\label{ejem:funcaracteristica}
   Let $(X,\mathcal{F}_X)$, define the characteristic function $\func{\delta}{X\times \mathcal{F}_X}{\R}$ by 
   \[
        \delta(x, F_X) = 
        \begin{cases}
            1 \quad & \text{si} \quad x\in F_X, \\
            0 \quad & \text{si} \quad x\notin F_X.
        \end{cases}
   \]
  for all $x\in X$ and $F_X\in\mathcal{F}_X$. 
   
  Note that the function $\delta_x=\delta(x,\cdot)$ corresponds to the Dirac measure at the point $x\in X$. 
   
    On the other hand, it is easy to see that $\delta(\cdot,F_X)$ for each $F_X\in\mathcal{F}_X$ is a function $\mathcal{F}_X/\mathcal{B}(\R)-$ meausurable. 
\end{ejemplo}

\begin{ejemplo}[Deterministic stochastic functions]
  Let $(X,\mathcal{F}_X),(Y,\mathcal{F}_Y)$ be measurable spaces and $\func{f}{(X,\mathcal{F}_X)}{(Y,\mathcal{F}_Y)}$ a $\mathcal{F}_X/\mathcal{F}_Y-$measurable  function. We define the stochastic function $\func{\delta_f}{(X,\mathcal{F}_X)}{(Y,\mathcal{F}_Y)}$ given by  
   \[
        \funcion{\delta_f}{X\times \mathcal{F}_Y}{[0,1]}{(x,F_Y)}{
        \delta_f(x,F_Y) = 
        \begin{cases}
            1 \quad &\text{si} \quad f(x) \in F_Y, \\
            0 \quad &\text{si} \quad f(x) \in F_Y^C. 
        \end{cases}
        }
   \]
   We will call $ \delta_f $ the deterministic stochastic function induced by $f$.
\end{ejemplo}

\begin{definicion}[Composition of stochastic functions]\label{def:composicionstoch}
    Let $\func{T}{(X,\mathcal{F}_X)}{(Y,\mathcal{F}_Y)}$ and $\func{U}{(Y,\mathcal{F}_Y)}{(Z,\mathcal{F}_Z)}$ two stochastic functions. The composition $\func{U\circ T}{(X,\mathcal{F}_X)}{(Z,\mathcal{F}_Z)}$ of $U$ and $T$ is defined by the integral  
     \[
         U\circ T (x,F_Z) = \int_Y U(\cdot,F_Z) \,d T_x = \displaystyle\int_{y\in Y} U(y,F_Z) \, T(x,dy). 
     \]
     for each $x\in X$ and $F_Z\in \mathcal{F}_Z$. 
\end{definicion}

\begin{ejemplo}
       Let $\func{f}{(X,\mathcal{F}_X)}{(Y,\mathcal{F}_Y)}$ and $\func{g}{(Y,\mathcal{F}_Y)}{(Z,\mathcal{F}_Z)}$ measurable functions. To exemplify the composition of two stochastic functions, let us consider $\delta_f$ y $\delta_g$ lthe deterministic stochastic functions induced by $f$ y $g$, respectively. Let $x\in X$ and $F_Z\in\mathcal{F}_Z$, we have
        \begin{equation}\label{eqn:ejemplo11}
            \delta_g\circ \delta_f (x,F_Z) = \int_Y \delta_g (\cdot,F_Z) \,  d\delta_{f,x}.
        \end{equation}
       On the other hand, let us note that
        \[
        \delta_g(x,F_Z) = 
        \begin{cases}
            1 \quad &\text{si} \quad g(x) \in F_Z, \\
            0 \quad &\text{si} \quad g(x) \in F_Z^C. 
        \end{cases}
        = 
        \begin{cases}
            1 \quad &\text{si} \quad x \in g^{-1}(F_Z), \\
            0 \quad &\text{si} \quad x \in g^{-1}\left(F_Z^C\right). 
        \end{cases}
        \]
        Thus, together with \eqref{eqn:ejemplo11}, we have to 
        \begin{align*}
            \delta_g\circ \delta_f (x,F_Z)
            & = \int_{g^{-1}(F_Z)} \delta_g (\cdot,F_Z) \,  d\delta_{f,x} + \int_{g^{-1}(F_Z^C)} \delta_g (\cdot,F_Z) \,  d\delta_{f,x}, \\
            & = \int_{g^{-1}(F_Z)}  \,  d\delta_{f,x} , \\
            & = \delta_f(x,g^{-1}(F_Z)), 
        \end{align*}
       where, we can conclude that 
        \[
             \delta_g\circ \delta_f (x,F_Z) = 
             \begin{cases}
                 1 \quad & \text{si} \quad f(x) \in g^{-1}(F_Z), \\
                 0 \quad & \text{si} \quad f(x) \in g^{-1}(F_Z^C).
             \end{cases}
             = 
             \begin{cases}
                 1 \quad & \text{si} \quad g(f(x)) \in F_Z, \\
                 0 \quad & \text{si} \quad g(f(x)) \in F_Z^C.
             \end{cases}
        \]
        that is, we have proven that
        \[
            \delta_f\circ \delta_g = \delta_{g\circ f}. 
        \]
\end{ejemplo}

\begin{observacion}
   We can see that the integral of~\eqref{def:composicionstoch} is well defined, since $ U(\cdot,F_Z)$ is a measurable function and $T_x$ is a probability measure defined on $(Y,\mathcal{F}_Y)$. 
     
    Subtract us to verify that $ U \circ T $ is a stochastic function.
    \begin{itemize}
        \item 
             Let $x\in X$,let us prove that  $U\circ T(x,\cdot)$ is a probability measure on $(Z,\mathcal{F}_Z)$. To do this, we will show that 
     
     \begin{itemize}
        \item
            $0\leq (U\circ T)_x (F_Z) \leq 1$ for each $F_Z\in\mathcal{F}_Z$. 
            
            Let $F_Z\in\mathcal{F}_Z$ and $y\in Y$, since $U_y$ is a probability measure on $(Z,\mathcal{F}_Z)$, we have
            \[
                0\leq U(y,F_Z) \leq 1, 
            \]
             whence, by the monotony of the integral, it follows that
             \[
                0\leq U\circ T (x,F_Z) \leq 1. 
             \]
         \item 
            $U\circ T(x,Z) = 1$. In fact, we have 
            \begin{align*}
                U\circ T (y,Z)
                & = \int_{Y} U(y,Z) \, T(x,dy), \\
                & = \int_Y \, T(x,dy), \\
                & = T(x,Y) , \\
                & = 1. 
            \end{align*}
        \item
            Let $\{F_n\}_{n\in\mathbb{N}}\subset\mathcal{F}_Z$ a collection of events of $\mathcal{F}_Z$ disjoint two by two, we must prove that  $U\circ T\left(x, \displaystyle\bigcup_{n\in\mathbb{N}} F_n \right) = \displaystyle\sum_{n\in\mathbb{N}} U\circ T(x,F_n)$. 
            
             We have to 
            \begin{align*}
                (U\circ T)\left(x, \displaystyle\bigcup_{n\in\mathbb{N}} F_n \right)
                & = \int_{Y} U \left(y, \displaystyle\bigcup_{n\in\mathbb{N}} F_n \right) \, d T_x,  \\ 
                 & = \int_Y \sum_{n\in\mathbb{N}} U(y,F_n) \, T(x,dy), 
            \end{align*}
            from where, thanks to the monotone convergence theorem, we obtain that 
            \begin{align*}
                (U\circ T)\left(x, \displaystyle\bigcup_{n\in\mathbb{N}} F_n \right)
                & = \sum_{n\in\mathbb{N}} \int_Y U(y,F_n) \, T(x,dy), \\
                & = \sum_{n\in\mathbb{N}} U\circ T(x,F_n).
            \end{align*}
     \end{itemize}
     With this, we have proved that  $U\circ T(x,\cdot)$ is a probability measure defined on $(Z,\mathcal{F}_Z)$. 
        \item
            Let $F_Z\in\mathcal{F}_Z$,we are going to show that $\func{U\circ T (\cdot,F_Z)}{X}{[0,1]}$ is $\mathcal{F}_X/\mathcal{B}([0,1])-$ measurable.
            
            Let's proof the result for characteristic functions, for this, let $\func{f}{(Y,\mathcal{F}_Y)}{(Z,\mathcal{F}_Z)}$ a  $\mathcal{F}_X/\mathcal{F}_Y-$measurable function, consider the deterministic stochastic function  $\func{\delta_f}{(Y,\mathcal{F}_Y)}{(Z,\mathcal{F}_Z)}$. For $x\in X$ and $F_Z\in\mathcal{F}_Z$, we have  
            \begin{align*}
                \delta_f\circ T (x,F_Z) 
                & = \int_Y \delta_f(\cdot,F_Z) \, dT_x , \\
                & = \int_{f^{-1}(F_Z)} \delta_f(\cdot,F_Z) \, dT_x + \int_{f^{-1}(F_Z^C)} \delta_f(\cdot,F_Z) \, dT_x, \\
                & = \int_{f^{-1}(F_Z)} \delta_f(\cdot,F_Z) \, dT_x, \\
                & = \int_{f^{-1}(F_Z)}  \, dT_x, \\
                & = T(x,f^{-1}(F_Z)). 
            \end{align*}
            With this, we have to
            \[
                \delta_f\circ T (\cdot,F_Z) = T(\cdot,f^{-1}(F_Z)), 
            \]
            from where, since  $T_{F_Y}$ is $\mathcal{F}_X / \mathcal{B}(\R)-$ measurable for each $F_Y\in\mathcal{F}_Y$ it follows that  $\delta_f\circ T (\cdot,F_Z)$ is too. 
            
             Thanks to the linearity of the integral, the previous result is still true if $ U_{F_Z} $ is a simple function. Then, the general case is a consequence of the monotone convergence and the fact that $ U_{F_Z} $ can be seen as the limit of an increasing sequence of simple functions. 
    \end{itemize}
\end{observacion}

\begin{ejemplo}[Category \textbf{Stoch}]
     
     The \textbf{Stoch} category is defined by the following information: 
     
     \begin{enumerate}
         \item 
            \textbf{Objects:} Measurable spaces;
        \item
            \textbf{Arrows:} Let $(X,\mathcal{F}_X)$ and $(Y,\mathcal{F}_Y)$ two measurable spaces, the set of arrows between these spaces is defined by
            \[
                \textbf{Stoch}((X,\mathcal{F}_X),(Y,\mathcal{F}_Y)) = \{U|\func{U}{(X,\mathcal{F}_X)}{(y,\mathcal{F}_Y)}\;\text{ is a stochastic function}\}.
            \]
        \item
            \textbf{Composition}: Composition of stochastic functions in the sense of the definition~\ref{def:composicionstoch}.
        \item
            \textbf{Identities}: The stochastic identity function is the characteristic function defined in example ~\ref{ejem:funcaracteristica} and we will denote it by $\func{Id}{(X,\mathcal{F}_X)}{(X,\mathcal{F}_X)}$.
     \end{enumerate}
    \end{ejemplo}
    
    \begin{observacion}
       \textbf{Stoch} is a category. To prove this, thanks to the previous results, it is enough to show that the composition in \textbf{Stoch} is associative and the identity axioms are satisfied.
        \begin{itemize}
            \item 
                \textbf{Associativity:}

                Let $\func{T}{(X,\mathcal{F}_X)}{(Y,\mathcal{F}_Y)}$, $\func{U}{(Y,\mathcal{F}_Y)}{(Z,\mathcal{F}_Z)}$ y $\func{V}{(Z,\mathcal{F}_Z)}{(W,\mathcal{F}_W)}$ stochastic functions, we will show that 
                \[
                    V\circ(U\circ T) = (V\circ U)\circ T.
                \]
                To do this, let  $x\in X$ and $F_W\in\mathcal{F}_W$, we must prove that 
                \[
                    V\circ(U\circ T)(x,F_W) = (V\circ U)\circ T (x,F_W). 
                \]
                
                On the one hand, for $y\in Y$ and $F_W\in\mathcal{F}_W$, we have 
                \[
                    V\circ U (y,F_W) = \int_Z V(\cdot,F_W) \, d U_y, 
                \]
                then
                \begin{align}\label{eqn:stochcat_01}
                    (V\circ U)\circ T (x,F_W)
                    & = \int_Y V\circ U(\cdot,F_W) \, d T_x, \notag \\
                    & = \int_Y \left( \int_Z V(\cdot,F_W) \, d U_y \right) \, dT_x
                \end{align}
                
               On the other hand, for $ x \in X $ y $ F_Z \in \mathcal{F}_Z $, we have
                \[
                     U\circ T (x,\mathcal{F}_Z) = \int_Y U(\cdot,F_Z) \,d T_x, 
                \]
               with which, we obtain that
                \[
                    V\circ (U\circ T) (x,F_W) = \int_Z V(\cdot,F_W) \, d( U\circ T)_x, 
                \]
                from where, together with proposition 3 of the lemma ~\ref{lema:monadaGirylema}, we obtain that 
                \[
                     V\circ (U\circ T) (x,F_W) =\int_Y \left( \int_Z V(\cdot,F_W) \, U_y \right) \, dT_x.
                \]
                Thus, combining the preceding identity with \eqref{eqn:stochcat_01}, we can conclude that  
                \[
                    V\circ(U\circ T)(x,F_W) = (V\circ U)\circ T (x,F_W). 
                \]

                \item 
                Let $\func{T}{(X,\mathcal{F}_X)}{(Y,\mathcal{F}_Y)}$ and $\func{U}{(Y,\mathcal{F}_Y)}{(X,\mathcal{F}_X)}$ stochastic functions, we must prove that
                
                \[
                     Id\circ U = U\quad\text{y}\quad T\circ I = T. 
                \]
                
                For this, let $y\in Y$ and $F_X\in\mathcal{F}_X$, we are going to show that 
                
                \[
                     Id\circ U (y,F_X) = U (y,F_X). 
                \]
                
                By definition of composition applied to the stochastic functions $ Id $ and $ U $, we have
                \begin{align*}
                    Id\circ U (y,F_X)
                    & = \int_X Id (x,F_X) \, d U_y, \\
                    & = \int_X \chi(x,F_X) \, d U_y, \\
                    & = \int_{F_X} \chi(x,F_X) \, d U_y + \int_{F_X^C} \chi(x,F_X) \, d U_y, \\
                    & = \int_{F_X}  \, d U_y, \\
                    & = U(y,F_X). 
                \end{align*}

                Now, let's show that  
                \[
                    T\circ Id = T.
                \]
                
                Let $x\in X$ and $F_Y\in\mathcal{F}_Y$, we must prove that 
                \[
                    T\circ Id (x,\mathcal{F}_X) = T (x,\mathcal{F}_X).
                \]
                
                Again, by definition of composition applied to the stochastic functions $ T $ and $ Id $, we have to
               \begin{align*}
                    T\circ Id (x,\mathcal{F}_X)
                    & =  \int_X T(x,F_X) \, d Id_x, \\ 
                    & = \int_X T(x,F_X) \, d \delta_x,
               \end{align*}
                from where, since $ \delta_x $ is the Dirac measure at the point $ x \in X $, we obtain that
                \[
                    T\circ Id (x,\mathcal{F}_X) = T(x,F_X), 
                \]
               with this, we can conclude that
                \[
                    T\circ Id = T.
                \]
                    \end{itemize}
    \end{observacion}

\begin{definicion}[Isomorphism]
    Let $\mathcal{C}$ be a category and $A,B\in\textbf{Ob}(\mathcal{C})$ objects. An arrow $\func{f}{A}{B}$  in $\mathcal{C}$ is an \emph{isomorphism} if there is an arrow  $\func{g}{B}{A}$ such that
    \[
        g\circ f = 1_A\quad\text{y}\quad f\circ g = 1_B. 
    \]
\end{definicion}

\begin{observacion}
   Given the arrow  $\func{f}{A}{B}$,  the arrow $\func{g}{B}{A}$ is unique. Thus, $g=f^{-1}$ is written. In this case we say that $ A $ and $ B $ are isomorphic, denoted by $ A \cong B $, if there is an isomorphism between them. 
\end{observacion}

\begin{ejemplo}
  We will present an example of isomorphism in the category \textbf{CMS}, which was presented in observation~\ref{obs:catCMS}. Let $(X,\mathcal{F}_X,\mu)$  be a measured space and $ c> 0 $, we have that $\mu$ is a measure defined on $(X,\mathcal{F}_X)$, from where $c\cdot\mu$ is too. Thus, consider the arrow  $\func{f}{(X,\mathcal{F}_X,\mu)}{(X,\mathcal{F}_X,c\cdot\mu)}$ in \textbf{CMS} defined by 
   \[
        \funcion{Id}{(X,\mathcal{F}_X)}{(X,\mathcal{F}_X)}{x}{x}
   \]
   We have that $ f $ is an isomorphism. Indeed, let's define $\func{g}{(X,\mathcal{F}_X,c\cdot\mu)}{(X,\mathcal{F}_X,\mu)}$ by 
   \[
        \funcion{Id}{(X,\mathcal{F}_X)}{(X,\mathcal{F}_X)}{x}{x.}
   \]
   Let's prove that $ g $ is a bounded function, for this, let $F_X\in\mathcal{F}_X$, since $Id^{-1}(F_X) = F_X$, it follows that 
   \[
        c\cdot \mu(Id^{-1}(F_X)) = c\cdot \mu(F_X) \leq 2 c \cdot\mu(F_X), 
   \]
   with which, taking $ M = 2 c> 0 $ the result follows.
   
   Furthermore, it is easy to see that $g\circ f = Id_{(X,\mathcal{F}_X,\mu)}$ and $f\circ g = Id_{(X,\mathcal{F}_X,c\cdot\mu)}$
   
\end{ejemplo}

The following definition allows us to consider new categories from those previously presented.

\begin{definicion}[Opposite Category]
    Let $\mathcal{C}$ be a category, its category \emph{opposite} or \emph{dual}, denoted by $\mathcal{C}^{\text{op}}$, is defined by the following information: 
    \begin{itemize}
        \item 
            \textbf{Objects:} \textbf{Ob}$(\mathcal{C}^{\text{op}}):=$ \textbf{Ob}($\mathcal{C}$), 
        \item
            \textbf{Arrows}: Let $A,B\in\textbf{Ob}(\mathcal{C}^\text{op})$, the collection of arrows between these objects is defined by $\mathcal{C}^{\text{op}}(B,A) = \mathcal{C}(A,B)$.
            
           That is, $\func{f}{B}{A}$ is an arrow in $\mathcal{C}^{\text{op}}$ if $\func{f}{A}{B}$ is an arrow in $\mathcal{C}.$ To refer to an arrow in the opposite category we will use the notation $f^{\text{op}}$.
        \item
            \textbf{Composition}: Let $A,B,C\in\textbf{Ob}(\mathcal{C}^\text{op})$, let's consider the arrows $\func{f^{\text{op}}}{B}{A}$ and $\func{g^{\text{op}}}{C}{B}$, its composition is defined by
            \[
                f^{\text{op}}\circ g^{\text{op}} = (g\circ f)^{\text{op}}.
            \]
        \item
            \textbf{Identities:} The same identities defined in $\mathcal{C}$. 
    \end{itemize}
\end{definicion}

\begin{observacion}
   Let us consider the following diagram in $\mathcal{C}$: 

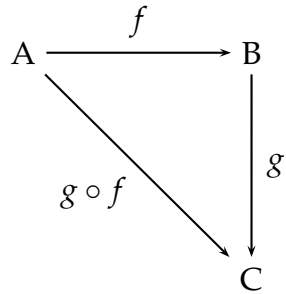
     \begin{figure}[h]
     \centering
     \begin{pspicture}[showgrid=false](-1,0)(3.5,3.5)
                    \rput(3,3){\rnode{B}{B}}
                    \rput(0,3){\rnode{A}{A}}
                    \rput(3,0){\rnode{C}{C}}
                    \ncline[nodesep=4pt]{->}{A}{B}\naput{$f$}
                    \ncline[nodesep=4pt]{->}{B}{C}\naput{$g$}
                    \ncline[nodesep=4pt]{->}{A}{C}\nbput{$g\circ f$}
                \end{pspicture}
     \end{figure}

     The corresponding diagram in $\mathcal{C}^{\text{op}}$ is

     
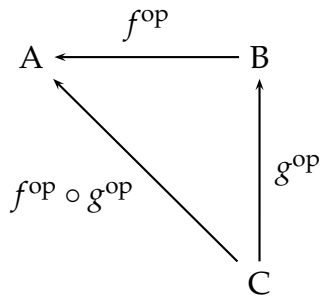
\begin{figure}[h]
     \centering
     \begin{pspicture}[showgrid=false](-1,0)(3.5,3.5)
                    \rput(3,3){\rnode{B}{B}}
                    \rput(0,3){\rnode{A}{A}}
                    \rput(3,0){\rnode{C}{C}}
                    \ncline[nodesep=4pt]{<-}{A}{B}\naput{$f^{\text{op}}$}
                    \ncline[nodesep=4pt]{<-}{B}{C}\naput{$g^{\text{op}}$}
                    \ncline[nodesep=4pt]{<-}{A}{C}\nbput{$f^{\text{op}}\circ g^{\text{op}}$}
                \end{pspicture}
     \end{figure}
\end{observacion}

\begin{definicion}[Subcategory]\label{def:subcat}
    Let $\mathcal{C}$ a category. A \emph{subcategory} $\mathcal{D}$  of $\mathcal{C}$ consists of a sub-collection $\textbf{Ob}(\mathcal{D})$ of $\textbf{Ob}(\mathcal{C})$ together with, for each  $A,B\in\textbf{Ob}(\mathcal{D})$, a subcollection $\mathcal{D}(A,B)$ of $\mathcal{C}(A,B)$, such that
    \begin{enumerate}
        \item 
            For each object  $A\in\textbf{Ob}(\mathcal{D})$, the identity arrow $\textbf{1}_A$ is in $\mathcal{D}$, that is, $\mathcal{D}$ is closed under identities.  
        \item
            For each pair of composable arrows $f\in \mathcal{D}(A,B)$ and $g\in \mathcal{D}(B,C)$, we have that $g\circ f$ is in $\mathcal{D}$, that is, $\mathcal{D}$ is closed under composition. 
    \end{enumerate}
    Also, $\mathcal{D}$ is said a full subcategory if 
    \[
        \mathcal{D}(A,B) = \mathcal{C}(A,B), 
    \]
   for all $A,B\in\textbf{Ob}(\mathcal{D})$. 
\end{definicion}

\begin{observacion}
   Consider the category $ \textbf{Set}_1 $ whose objects with sets and whose arrows between sets are injective functions. It's easy to see that $ \textbf{Set}_1 $ is a subcategory of \textbf{Set}.
\end{observacion}

\begin{ejemplo}
  The \textbf{CPS} category of probability spaces and bounded functions, defined in  \ref{def:catCPS}, is a full subcategory of the \textbf{CMS} category of measured spaces and bounded functions. It is easy to see that the conditions of the definition~\ref{def:subcat} are satisfied and that $\textbf{CMS} ( (X,\mathcal{F}_X,\mathbb{P}_X),(Y,\mathcal{F}_Y,\mathbb{P}_Y)  ) = \textbf{CPS}((X,\mathcal{F}_X,\mathbb{P}_X),(Y,\mathcal{F}_Y,\mathbb{P}_Y) )$ for each  $(X,\mathcal{F}_X,\mathbb{P}_X)$, $(Y,\mathcal{F}_Y,\mathbb{P}_Y)\in\textit{\textbf{Ob}}(\textbf{CPS})$. 
\end{ejemplo}

In many of the categories that we have defined, the collection of all objects is too large to form a set. This topic will not be deepened, since it is not the objective of this work; however, consider the following definition:

\begin{definicion}
   A $ \mathcal{C} $ category is named \emph{small} if both the $ \textbf{Ob} (\mathcal{C}) $ collection of $ \mathcal{C} $ objects and the $ \boldsymbol{\mathcal{A}} (\mathcal{C}) $ of arrows of $ \mathcal{C} $ are sets. Otherwise $ \mathcal{C} $ is called  \emph{large}.
    
     A category $ \mathcal{C} $ is \emph{locally small} if for every $ A, B \in\textbf{Ob} (\mathcal{C}) $, the collection $ \mathcal{C} (A, B) $ is a set. 
\end{definicion}

\begin{observacion}
    If a category is small, then it is locally small.
\end{observacion}

\begin{observacion}
     We will denote by \textbf{Cat} the category of all small categories and by \textbf{LCat} the category of all locally small categories.
\end{observacion}

\begin{ejemplo}
    \begin{enumerate}
        \item 
           The \textbf{Set} category of sets is a large category.
        \item
          The \textbf{CPS} s a locally small category. Indeed, for each $\overline{X}=(X,\mathcal{F}_X,\mathbb{P}_X)$ and $\overline{Y}=(Y,\mathcal{F}_Y,\mathbb{P}_Y)\in\textit{\textbf{Ob}}(\textbf{CPS})$, we have that 
            \[
                \textbf{CPS}(\overline{X},\overline{Y}) \subset Y^X, 
            \]
            where $Y^X$ represents the set of functions from the set $X$ to the set $Y$. 
        \item
            Similarly, \textbf{Set} is a locally small category
    \end{enumerate}
\end{ejemplo}

\begin{definicion}
    Let $\mathcal{C}$ be a category and $I$, $T\in\textbf{Ob}(\mathcal{C})$ objects. $ I $ is said to be \emph{initial} if, for every $A\in\textbf{Ob}(\mathcal{C})$, there is exactly one arrow $\func{f}{I}{A}$ in $\mathcal{C}$. 
    
    On the other hand, $ T $ is said \emph{terminal} if, for every, $A\in\textbf{Ob}(\mathcal{C})$, there is exactly one arrow $\func{f}{A}{T}$ in $\mathcal{C}$.
\end{definicion}

\begin{observacion}
    We see that the empty set is an initial object of \textbf{Set}. Also, any set with an element, that is, a singlet, is a terminal object. Indeed, for each $A\in\textit{\textbf{Ob}}(\textbf{Set})$, there is a unique function such that 
    \[
        \funcion{f}{A}{\{*\}}{x}{f(x)=*.}
    \]
\end{observacion}

\section{Funtor}

A functor is an arrow between categories. Formally, we have the following definition:

\begin{definicion}[Functor]\label{def:functor}
   Let $\mathcal{C}$ and $\mathcal{D}$ two categories, a \emph{covariant funtor} $\func{F}{\mathcal{C}}{\mathcal{D}}$ it consists of: 
    \begin{itemize}
        \item 
            A function 
            \[
                \funcion{F}{\textbf{Ob}(\mathcal{C})}{\textbf{Ob}(\mathcal{D})}{A}{F(A),}
            \]
        \item
            For all $A,B\in\textbf{Ob}(\mathcal{C})$, a function
            \[
                \funcion{F}{\mathcal{C}(A,B)}{\mathcal{D}(F(A),F(B))}{f}{F(f).}
            \]
            such that
            \[
                F(g\circ f) = F(g)\circ F(F)\quad\text{y}\quad F(\textbf{1}_A) = \textbf{1}_{F(A),}
            \]
           for any pair of composable arrows $f,g\in\boldsymbol{\mathcal{A}}(\mathcal{C})$ and for any object $A\in\textbf{Ob}(\mathcal{C})$. 
            
            On the other hand, if $\func{F}{\mathcal{C}}{\mathcal{D}}$ is such that 
            \[
                F(g\circ f) = F(f)\circ F(g)\quad\text{y}\quad F(\textbf{1}_A) = \textbf{1}_{F(A),}
            \]
           then $ F $ is said to be a \emph{contravariant functor}. 
    \end{itemize}
\end{definicion}

\begin{observacion}
    Note that a contravariant functor can be defined as a covariant functor in the opposite category $\mathcal{C}$. That is, if  $\func{F}{\mathcal{C}}{\mathcal{D}}$ eis a contravariant functor, then we will write  $\func{F}{\mathcal{C}^{\text{op}}}{\mathcal{D}}$ and we'll call it functor. Thus, we will not make a distinction between covariant and contravariant functors. 
\end{observacion}

Again, we will first introduce classic functor examples. 

\begin{ejemplo}
    \begin{itemize}
        \item 
            Let $\func{F}{\mathcal{C}}{\mathcal{B}}$ y $\func{G}{\mathcal{B}}{\mathcal{D}}$ two funtors, the composition funtor $ G \circ F $ is defined by: 
            \begin{itemize}
                \item 
                    For objects: 
                    \[
                        \funcion{G\circ F}{\textit{\textbf{Ob}}(\mathcal{C})}{\textit{\textbf{Ob}}(\mathcal{D})}{A}{G(F(A)).}
                    \]
                \item
                    For arrows: 
                    \[
                        \funcion{G\circ F}{\boldsymbol{\mathcal{C}}(A,B)}{\boldsymbol{\mathcal{D}}(G(F(A)),G(F(B)))}{[\func{f}{A}{B}]}{[\func{G(F(f))}{G(F(A))}{G(F(B))}].}
                    \]
            \end{itemize}

        \item
            Let $\mathcal{C}$ be a category, the identity functor is defined, denoted by $\func{\textbf{1}_{\mathcal{C}}}{\mathcal{C}}{\mathcal{C}}$, as follows 
            \[
                \textbf{1}_{\mathcal{C}}(A) = A\quad\text{y}\quad \textbf{1}_{\mathcal{C}}(f) = f,  
            \]
            for every object $A\in\textit{\textbf{Ob}}(\mathcal{C})$nd for every arrow $f\in\boldsymbol{\mathcal{A}}(\mathcal{C})$.
        \item
           Let $ A, B $ be two preordered sets, a covariant functor between the corresponding categories is exactly an increasing function.
    \end{itemize}
\end{ejemplo}

The preceding example allows us to present the following observation:

\begin{observacion}
    Consider the category \textbf {Cat} and the category \textbf {LCat}, the composition of functors in these categories is defined as the composition functor and the identity arrows correspond to the identity functors.  
\end{observacion}

On the other hand, examples of functors in the probabilistic context would be:

\begin{ejemplo}
   \begin{itemize}
       \item 
             Consider the category \textbf {CPS} and the category \textbf{Meas}, the forgetting functor is defined $\func{U}{\textbf{CPS}}{\textbf{Meas}}$ by:
            \begin{enumerate}
                \item 
                    For objects: 
                    \[
                        \funcion{\textbf{U}}{\textit{\textbf{Ob}}(\textbf{CPS})}{\textit{\textbf{Ob}}(\textbf{Meas})}{(X,\mathcal{F}_X,\mathbb{P}_X)}{(X,\mathcal{F}_X).}
                    \]
                \item
                    Let $\overline{X} = (X,\mathcal{F}_X,\mathcal{P}_X)$ and $\overline{Y} = Y,\mathcal{F}_Y,\mathcal{P}_Y)$ probability spaces, for arrows:
                    \[
                        \funcion{\textbf{U}}{\textbf{CPS}(\overline{X},\overline{Y})}{\textbf{Meas}((X,\mathcal{F}_X),(Y,\mathcal{F}_Y))}{[\func{f}{\overline{X}}{\overline{Y}}]}{[\func{f}{(X,\mathcal{F}_X)}{(Y,\mathcal{F}_Y)}].}
                    \]
                    In this case, the functor sends a probability space  $\overline{X}$ in a measurable space  $(X,\mathcal{F}_X)$ forgetting its probability measure.
                    
                     In general, it is not possible to formally define a forgetting functor; however, the idea of this functor is that \textit{forgets} some structure with respect to the domain category objects and arrows. 
            \end{enumerate}    
        \item
            The canonical normalization functor is defined between the categories \textbf{CMS} and \textbf{CPS}  $\func{\mathcal{N}}{\textbf{CMS}}{\textbf{CPS}}$ by: 
            \begin{enumerate}
                \item 
                    For objects: 
                    \[
                        \funcion{\mathcal{N}}{\textit{\textbf{Ob}}(\textbf{CMS})}{\textit{\textbf{Ob}}(\textbf{CPS})}{(X,\mathcal{F}_X,\mu)}{\left(X,\mathcal{F}_X,\frac{1}{\mu(X)} \mu\right).}
                    \]
                \item
                    Let $\overline{X} = (X,\mathcal{F}_X,\mu)$ and $\overline{Y} = (Y,\mathcal{F}_Y,\nu)$ finite measured spaces, for arrows:
                    \[
                        \funcion{\mathcal{N}}{\textbf{CMS}(\overline{X},\overline{Y})}{\textbf{CPS}\left(\left(X,\mathcal{F}_X,\frac{1}{\mu(X)}\mu\right),\left(Y,\mathcal{F}_Y,\frac{1}{\nu(Y)}\nu\right)\right)}{[\func{f}{\overline{X}}{\overline{Y}}]}{\left[\func{\mathcal{N}(f)}{\left(X,\mathcal{F}_X,\frac{1}{\mu(X)}\mu\right)}{\left(Y,\mathcal{F}_Y,\frac{1}{\nu(Y)}\nu\right)}\right].}
                    \]
                    We see that the functor  $\mathcal{N}$ sends a finite measured space in a probability space. Also, note that $\mathcal{N}(f)$ s a bounded function. Indeed, since $ f $ is a bounded function, we know that $ M> 0 $ exists such that $\mu(f^{-1}(F_Y))\leq M \nu(F_Y)$ for all  $F_Y\in\mathcal{F}_Y$. Thus, taking  $M_1 = \frac{\nu(Y)}{\mu(X)} M >0$, Thus, taking 
                    \[
                        \frac{1}{\mu(X)}\mu(f^{-1}(F_Y))\leq M_1 \frac{1}{\nu(Y)}\nu(F_Y),
                    \]
                   for all $F_Y\in\mathcal{F}_Y$.
            \end{enumerate}    
   \end{itemize}
\end{ejemplo}

We will introduce the Giry functor defined in ~\cite{GiryM}, for this, let's remember the \textbf{Meas} category of measurable spaces and measurable functions.

\begin{definicion}[Giry functor]\label{defi:FGiry}
   Let $(X,\mathcal{F}_X)$ and $(Y,\mathcal{F}_Y)$ measurable spaces. The Giry functor $\func{\mathcal{P}}{\textbf{Meas}}{\textbf{Meas}}$ is defined by: 
    \begin{enumerate}
        \item
            For objects: 
            \[
                \funcion{\mathcal{P}}{\textit{\textbf{Ob}}(\textbf{Meas})}{\textit{\textbf{Ob}}(\textbf{Meas})}{(X,\mathcal{F}_X)}{ \mathcal{P}(X,\mathcal{F}_X) := (\textbf{P}(X),\textbf{P}(\mathcal{F}_X)), }
            \]
            where $\textbf{P}(X)$  is the set of probability measures defined on $X$, that is,
            \[
                \textbf{P}(X):=\{\func{\mathbb{P}_X}{(X,\mathcal{F}_X)}{(\R,\mathcal{B}(\R))}: \mathbb{P}_X\; \text{is a probability measure}\}.
            \]
            On the other hand, to define the $\sigma-$algebra on  $\textbf{P}(X)$, sea $F_X\in\mathcal{F}_X$, consider the evaluation function defined by 
            \[
                \funcion{i_{F_X}}{\textbf{P}(X)}{[0,1]}{\mathbb{P}_X}{i_{F_X}(\mathbb{P}_X) = \mathbb{P}_X(F_X).}
            \]
           Consider the smallest $\sigma-$algebra such that the family of functions $\{i_{F_X}\}_{F_X\in\mathcal{F}_X}$ is measurable. So, we have to 
            \begin{align*}
                \textbf{P}(\mathcal{F}_X) 
                & := \sigma\left( \{i_{F_X}^{-1}(B): F_X\in\mathcal{F}_X\; \text{ y } B\in\mathcal{B}(\R)\} \right), \\
                & = \sigma\left(\{ \{\mathbb{P}_X\in\textbf{P}(X): \mathbb{P}_X(F_X)\in B\}: F_X\in\mathcal{F}_X\; \text{ y } B\in\mathcal{B}(\R) \}
                \right).
            \end{align*}
            With which, we obtain the measurable space $(\textbf{P}(X),\textbf{P}(\mathcal{F}_X))$.
        \item
            For arrows: 
            \[
                \funcion{\mathcal{P}}{\textbf{Meas}((X,\mathcal{F}_X),(Y,\mathcal{F}_Y))}{\textbf{Meas}((\textbf{P}(X),\textbf{P}(\mathcal{F}_X)),(\textbf{P}(Y),\textbf{P}(\mathcal{F}_Y)))}{[\func{f}{(X,\mathcal{F}_X)}{(Y,\mathcal{F}_Y)}]}{[ \func{\mathcal{P}(f)}{(\textbf{P}(X),\textbf{P}(\mathcal{F}_X))}{(\textbf{P}(Y),\textbf{P}(\mathcal{F}_Y))}], }
            \]
            where $\textbf{P}(f)$ is a function such that 
            \[
                \funcion{\mathcal{P}(f)}{(\textbf{P}(X),\textbf{P}(\mathcal{F}_X))}{(\textbf{P}(Y),\textbf{P}(\mathcal{F}_Y))}{\mathbb{P}_X}{\mathbb{P}_X\circ f^{-1},}
            \]
            where $\mathbb{P}\circ f^{-1}$ is the image of the measure $\mathbb{P}_X$ with respect to the function $f$. 
    \end{enumerate}
\end{definicion}

\begin{proposicion}
   The Giry functor is well defined.
\end{proposicion}

\begin{proof}
    We will do the proof in two steps, first we will prove that for an arrow $\func{f}{(X,\mathcal{F}_X)}{(Y,\mathcal{F}_Y)}$ in \textbf{Meas}, the function $\func{\mathcal{P}(f)}{(\textbf{P}(X),\textbf{P}(\mathcal{F}_X))}{(\textbf{P}(Y),\textbf{P}(\mathcal{F}_Y))}$ is $\textbf{P}(\mathcal{F}_X)/\textbf{P}(\mathcal{F}_Y)-$ measurable. Next, we will prove that the Giry functor preserves compositions and identities. 
    
    \begin{enumerate}
        \item
            Let $A\in \{ i^{-1}_{F}(B) : F_Y\in\mathcal{F}_Y\; \text{ y }\; B\in\mathcal{B}(\R))\}$, we know there are $F_Y\in\mathcal{F}_Y$ and $B\in\mathcal{B}(\R)$ such that
            \[
                A = i^{-1}_{F}(B). 
            \]
            To prove that $\mathcal{P}(f)$ is $\textbf{P}(\mathcal{F}_X)/\textbf{P}(\mathcal{F}_Y)-$measurable, just show that $(\mathcal{P}(f))^{-1} ( i^{-1}_{F}(B)) \in \sigma(\{i_{F_X}^{-1}(B): F_X\in\mathcal{F}_X\; \text{ y } B\in\mathcal{B}(\R)\})$. Indeed, we have to 
            \begin{align*}
                (\mathcal{P}(f))^{-1} ( i^{-1}_{F_Y}(B))
                & = \{ \mathbb{P}_X\in \textbf{P}(X): \mathcal{P}(f)(\mathbb{P}_X) \in i^{-1}_{F_Y}(B) \}, \\
                & = \{ \mathbb{P}_X\in \textbf{P}(X): \mathbb{P}_X \circ f^{-1} \in i^{-1}_{F_Y}(B) \}, \\
                & = \left\lbrace \mathbb{P}_X\in \textbf{P}(X): i_{F_Y} \left(\mathbb{P}_X\circ f^{-1}\right) \in B  \right\rbrace, \\
                & = \left\lbrace \mathbb{P}_X\in \textbf{P}(X):  \mathbb{P}_X \left(f^{-1} (F_Y) \right)\in B \right\rbrace, \\
                & = i^{-1}_{f^{-1}(F_Y)}(B), 
            \end{align*}
           from where, given that $f^{-1}(F_Y)\in\mathcal{F}_X$, it follows that 
            \[
                 (\mathcal{P}(f))^{-1} ( i^{-1}_{F_Y}(B)) \in \sigma(\{i_{F_X}^{-1}(B): F_X\in\mathcal{F}_X\; \text{ y } B\in\mathcal{B}(\R)\}). 
            \]
                   
        \item
           
            \begin{itemize}
                \item 
                    \textbf{Associativity:}
                   Let $\func{f}{(X,\mathcal{F}_X)}{(Y,\mathcal{F}_Y)}$ and $\func{g}{(Y,\mathcal{F}_Y)}{(Z,\mathcal{F}_Z)}$ composable arrows in \textbf{Meas}, we are going to show that
            \[
                \mathcal{P}(g\circ f) = \mathcal{P}(g)\circ\mathcal{P}(f).
            \]
            
           On the one hand, let $\mathbb{P}_X\in\textbf{P}(X)$, aapplying the definition of $\mathcal{P}$ for the arrow $\func{g\circ f}{(X,\mathcal{F}_X)}{(Z,\mathcal{F}_Z)}$, we have  
            \begin{equation*}
                \mathcal{P}(g\circ f)(\mathbb{P}_X) = \mathbb{P}_X\circ (g\circ f)^{-1}. 
            \end{equation*}
            
            Now, let $F_Z\in\mathcal{F}_Z$, by definition of the image of the measure $\mathbb{P}_X$ with respect to the function $g\circ f$, we know that  
            \[
                \left(\mathbb{P}_X\circ (g\circ f)^{-1}\right) (F_Z) =\mathbb{P}_X\left( (g\circ f)^{-1} (F_Z) \right), 
            \]
            from where, we have to
            \begin{equation}\label{eqn:prop3.2}
                \mathbb{P}_X\left( (g\circ f)^{-1} (F_Z) \right) = \mathbb{P}_X \left( f^{-1}(g^{-1}(F_Z)) \right).
            \end{equation}
        
            On the other hand, it follows that 
            \[
                \mathcal{P}(f)(\mathbb{P}_X) = \mathbb{P}_X\circ f^{-1}, 
            \]
            with which, given that  $\mathbb{P}_X\circ f^{-1}\in\textbf{P}(Y)$, it follows that 
            \begin{equation*}
                \mathcal{P}(g) \left( \mathbb{P}_X\circ f^{-1} \right) = \left(\mathbb{P}_X\circ f^{-1}\right)\circ g^{-1}.  
            \end{equation*}
          Then, let $F_Z\in\mathcal{F}_Z$, by definition of the pushforward measure, we have that
            \begin{align*}
               \left( \left(\mathbb{P}_X\circ f^{-1}\right)\circ g^{-1}\right) (F_Z) = (\mathbb{P}_X\circ f^{-1})(g^{-1}(F_Z)),
            \end{align*}
            with this, again applying the  definition of the pushforward measure, we obtain that
            \begin{equation}\label{eqn:prop3.2_01}
                \left( \left(\mathbb{P}_X\circ f^{-1}\right)\circ g^{-1}\right) (F_Z) = \mathbb{P}_X \left( f^{-1}(g^{-1})(F_Z) \right).
            \end{equation}

            Combining \eqref{eqn:prop3.2} and \eqref{eqn:prop3.2_01}, we obtain that 
            \[
                \left(\mathbb{P}_X\circ (g\circ f)^{-1}\right) (F_Z) = \left( \left(\mathbb{P}_X\circ f^{-1}\right)\circ g^{-1}\right) (F_Z), 
            \]
            for each $F\in\mathcal{F}_Z$. 
            \item
                The laws of identity are proven in the same way as associativity. 
            \end{itemize}
    \end{enumerate}
\end{proof}

\section{Natural transformations}

Natural transformations are a notion of ``map between functors'' and it applies when the functors have the same domain and codomain. Formally, we have the following definition:

\begin{definicion}[Natural transformation]\label{def:transformacionatural}
    Let $\mathcal{C}$, $\mathcal{D}$ categories and $\func{F}{\mathcal{C}}{\mathcal{D}}$, $\func{G}{\mathcal{C}}{\mathcal{D}}$ functors. A \emph{natural transformation} $\func{\alpha}{F}{G}$ is a family \newline $\left(\func{\alpha_A}{F(A)}{G(A)}\right)_{A\in\textbf{Ob}(\mathcal{C})}$ of arrows in  $\mathcal{D}$ such that, for each arrow $\func{f}{A}{B}$ in $\mathcal{C}$, the diagram
    
    \begin{figure}[h]
    \centering
            \begin{pspicture}[showgrid=false](0,0)(3,4.5)
                \rput(0,4){\rnode{A}{$F(A)$}}
                \rput(3,4){\rnode{B}{$F(B)$}}
                \rput(0,1){\rnode{C}{$G(A)$}}
                \rput(3,1){\rnode{D}{$G(B)$}}
            \ncline[nodesep=4pt]{->}{A}{B}\naput{$F(f)$}
            \ncline[nodesep=4pt]{->}{C}{D}\nbput{$G(f)$}
            \ncline[nodesep=4pt]{->}{A}{C}\nbput{$\alpha_A$}
            \ncline[nodesep=4pt]{->}{B}{D}\naput{$\alpha_B$}
            \end{pspicture}
        \label{fig:transformacionN}
    \end{figure}
    
    commutes, that is, 
    \begin{equation}\label{eqn:naturalidad}
        \alpha_B \circ F(f) = G(f)\circ \alpha_A.
    \end{equation}
    
    The arrows $ \alpha_A $ are called the components of the natural transformation $ \alpha $.
    
\end{definicion}

\begin{observacion}
    The definition of natural transformation ~\ref{def:transformacionatural} states that for each arrow $\func{f}{A}{B}$ in $\mathcal{C}$ it is possible to construct exactly one arrow $\func{}{F(A)}{G(B)}$ in $\mathcal{D}$.
    
     Furthermore, the identity  \eqref{eqn:naturalidad} is called the naturality condition. 
\end{observacion}

We can see that natural transformations are a type of arrow, so we would expect to be able to compose them. To do this, let  $\mathcal{C},\mathcal{D}$ categories, $\func{F,G,H}{\mathcal{C}}{\mathcal{D}}$ functors and $\func{\alpha}{F}{G}$, $\func{\beta}{G}{H}$ natural transformations. The natural compound transformation  $\func{\beta\circ\alpha}{F}{H}$ is defined by 
\[
    (\beta\circ \alpha)_A = \beta_A\circ \alpha_A, 
\]
for each  $A\in\textit{\textbf{Ob}}(\mathcal{C})$. We can also consider the natural transformation identity $\func{1_{F}}{F}{F}$, given by 
\[
    (1_F)_A = 1_{F(A)},
\]
for each  $A\in\textit{\textbf{Ob}}(\mathcal{C})$.With this, for two categories $\mathcal{C}$ and $\mathcal{D}$, there is a category whose objects are the functors from category $\mathcal{C}$ to category  $\mathcal{D}$ and whose arrows between functors are natural transformations. We will denote this category by  $[\mathcal{C},\mathcal{D}]$ and call it the functor category from  $\mathcal{C}$ a $\mathcal{D}$.

\begin{ejemplo}\label{ejem:deltaDirac}
    Let $(X,\mathcal{F}_X)$ be a measurable space and the function $\alpha_{(X,\mathcal{F}_X)}$ defined by
    \[
        \funcion{\alpha_{(X,\mathcal{F}_X)}}{(X,\mathcal{F}_X)}{(\textbf{P}(X),\textbf{P}(\mathcal{F}_X))}{x}{\alpha_X(x) = \delta_x,} 
    \]
    where $\delta_x$ is the Dirac measure at the point  $x\in X$.
    
    Also, let $\func{1_{\textbf{Meas}}}{\textbf{Meas}}{\textbf{Meas}}$ be the identity functor in \textbf{Meas} and the Giry functor $\func{\mathcal{P}}{\textbf{Meas}}{\textbf{Meas}}$.
    
     We have that $\func{\alpha}{1_{\textbf{Meas}}}{\mathcal{P}}$ is a natural transformation. We will do the demonstration of the fact that $ \alpha $ is a natural transformation in two steps: 
    
    \begin{enumerate}
        \item 
            Let $(X,\mathcal{F}_X)\in\textit{\textbf{Ob}}(\textbf{Meas})$ and object in \textbf{Meas}, we are going to show that  \newline $\func{\alpha_{(X,\mathcal{F}_X)}}{(X,\mathcal{F}_X)}{(\textbf{P}(X),\textbf{P}(\mathcal{F}_X))}$ is a  $\mathcal{F}_X/\textbf{P}(\mathcal{F}_X)-$meausurable function. 
            
            For this, let $A\in \{ i^{-1}_{F_X}(B) : F_X\in\mathcal{F}_X\; \text{ y }\; B\in\mathcal{B}(\R))\}$, we know there are  $F_X\in\mathcal{F}_X$ and $B\in\mathcal{B}(\R)$ such that
            \[
                A = i^{-1}_{F_X}(B). 
            \]
             We have to
            \begin{align}\label{eqn:medtransnatural}
                \alpha_{(X,\mathcal{F}_X)}^{-1}(A)
                & = \alpha_{(X,\mathcal{F}_X)}^{-1} \left( i^{-1}_{F_X}(B) \right), \notag \\
                & = \{x\in X: \alpha(x) \in i^{-1}_{F_X}(B)  \}, \notag\\
                & = \{x\in X: i_{F_X}(\alpha_x) \in B \},\notag \\
                & = \{x\in X: \alpha_x (F_X) \in B \}. 
            \end{align}
            
           On the other hand, we know that
            \[
                \alpha_x (F_X) = 
                \begin{cases}
                    1 \quad & \text{si} \quad x\in F_X, \\
                    0 \quad & \text{si} \quad x\in F_X^C. 
                \end{cases}
            \]
            Thus, together with  \eqref{eqn:medtransnatural}, let's consider the following cases: 
            \begin{enumerate}
                \item 
                    If $\{1\}\in B$, then $\alpha_{(X,\mathcal{F}_X)}^{-1}(A) = F_X$. 
                \item
                    If $\{0\}\in B$, then $\alpha_{(X,\mathcal{F}_X)}^{-1}(A) = F_X^C$.
                \item
                    If $\{0,1\}\in B$, then $\alpha_{(X,\mathcal{F}_X)}^{-1}(A) = X$. 
                \item
                     If $\{0,1\}\notin B$, then $\alpha_{(X,\mathcal{F}_X)}^{-1}(A) = \emptyset$. 
            \end{enumerate}
            We can see that in all cases we have that $\alpha_{(X,\mathcal{F}_X)}^{-1}(A)\in\mathcal{F}_X$. Therefore, we can conclude that  $\func{\alpha_{(X,\mathcal{F}_X)}}{(X,\mathcal{F}_X)}{(\textbf{P}(X),\textbf{P}(\mathcal{F}_X))}$ is a $\mathcal{F}_X/\textbf{P}(\mathcal{F}_X)-$ measurable function. 
        \item

    Now, we will prove that the naturality condition is satisfied. Indeed, since $1_{\textbf{Meas}}(X,\mathcal{F}_X) = (X,\mathcal{F}_X)$ for every object $(X,\mathcal{F}_X)\in\textit{\textbf{Ob}}(\textbf{Meas})$ and $1_{\textbf{Meas}}(f) = f$  all arrow $f\in\boldsymbol{\mathcal{A}}(\textbf{Meas})$, it must be proved that the following diagram 

    \begin{figure}[h]
    \centering
     \psset{unit=1.3cm}
            \begin{pspicture}[showgrid=false](0,0.5)(3,4.5)
                \rput(0,4){\rnode{A}{$(X,\mathcal{F}_X)$}}
                \rput(3,4){\rnode{B}{$(Y,\mathcal{F}_Y)$}}
                \rput(0,1){\rnode{C}{$(\textbf{P}(X),\textbf{P}(\mathcal{F}_X))$}}
                \rput(3,1){\rnode{D}{$(\textbf{P}(Y),\textbf{P}(\mathcal{F}_Y))$}}
            \ncline[nodesep=4pt]{->}{A}{B}\naput{$f$}
            \ncline[nodesep=4pt]{->}{C}{D}\nbput{$\mathcal{P}(f)$}
            \ncline[nodesep=4pt]{->}{A}{C}\nbput{$\alpha_X$}
            \ncline[nodesep=4pt]{->}{B}{D}\naput{$\alpha_Y$}
            \end{pspicture}
    \end{figure}
    
    commutes, for which, we are going to verify that
    \[
        \alpha_Y\circ f = \mathcal{P}(f)\circ \alpha_X.
    \]
   Let $ x \in X $, on the one hand, we have
    \begin{align*}
         (\alpha_Y\circ f)(x)
         & = \alpha_Y (f(x)) \\
         & = \delta_{f(x)}, 
    \end{align*}
    from where, for each $ A \in \mathcal{F}_Y $, by definition of the Dirac measure at the point $ f (x) \in A $, we obtain that
    \begin{equation}\label{eqn:ejemplo11_01}
        \delta_{f(x)} (A)
        = 
        \begin{cases}
            1 \quad & \text{si} \quad f(x)\in A , \\
            0 \quad & \text{si} \quad f(x)\in A^C.
        \end{cases}
        = \begin{cases}
            1 \quad & \text{si} \quad x\in f^{-1}(A) , \\
            0 \quad & \text{si} \quad x\in f^{-1}(A^C).
        \end{cases}
    \end{equation}
    
    On the other hand, we have to
    \begin{align*}
         \left(\mathcal{P}(f)\circ \alpha_X\right)(x)
         & = \mathcal{P}(f)(\delta_x) , \\
         & = \delta_x\circ f^{-1}, 
    \end{align*}
    from where, for each $ A \in \mathcal{F}_Y $, it follows that
    \begin{equation}\label{eqn:ejemplo11_02}
        \left(\delta_{x} \circ f^{-1}\right) (A) = \delta_x(f^{-1}(A))
        = 
        \begin{cases}
            1 \quad & \text{if} \quad x\in f^{-1}(A) , \\
            0 \quad & \text{if} \quad x\in f^{-1}(A^C).
        \end{cases}
    \end{equation}
    
    Thus, together with \eqref{eqn:ejemplo11_01} and \eqref{eqn:ejemplo11_02}, we can conclude that 
    \[
        \alpha_Y\circ f = \mathcal{P}(f)\circ \alpha_X. 
    \]

    \end{enumerate}
\end{ejemplo}

\section{Products}

\begin{definicion}
    Let $\mathcal{C}$ be a category and  $A,B\in\textbf{Ob}(\mathcal{C})$ objects. A \emph{product} of $ A $ and $ B $ consists of: 
    \begin{itemize}
        \item 
            An object $P\in \textbf{Ob}(\mathcal{C})$; 
        \item
            A pair of arrows $\func{p_1}{P}{A}$, $\func{p_2}{P}{B}\in\boldsymbol{\mathcal{A}}(\mathcal{C})$; 
    \end{itemize}
   such that, for each object $C\in\textbf{Ob}(\mathcal{C})$ and for each pair of arrows $\func{f_1}{C}{A}$ and $\func{f_2}{C}{B}$, there is a single arrow $\func{\overline{f}}{C}{P}$ in $\mathcal{C}$ such that the following diagram: 
   
   \newpage
    
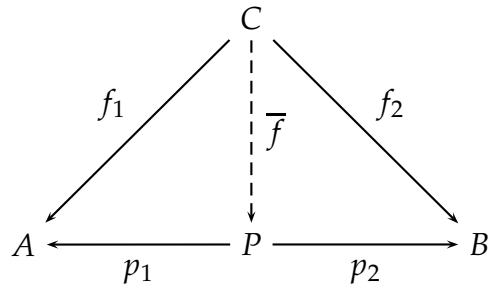
\begin{figure}[h]
    \centering
        \begin{pspicture}[showgrid=false](0,0)(6,4.5)
            \rput(0,1){\rnode{A}{$A$}}
            \rput(3,1){\rnode{B}{$P$}}
            \rput(6,1){\rnode{C}{$B$}}
            \rput(3,4){\rnode{D}{$C$}}
        \ncline[nodesep=4pt]{<-}{A}{B}\nbput{$p_1$}
        \ncline[nodesep=4pt]{<-}{C}{B}\naput{$p_2$}
        \ncline[nodesep=4pt]{<-}{A}{D}\naput{$f_1$}
        \ncline[nodesep=4pt]{<-}{C}{D}\nbput{$f_2$}
        \ncline[nodesep=4pt,linestyle=dashed]{<-}{B}{D}\nbput{$\overline{f}$}
        \end{pspicture}
    \end{figure}
    
    commutes, that is,
    \[
        p_1\circ \overline{f} = f_1 \quad\text{y}\quad   p_2\circ \overline{f} = f_2. 
    \]
    Also, the arrows $ p_1 $ and $ p_2 $ are called projections.
\end{definicion}

\begin{observacion}
    \begin{itemize}
        \item 
            Products don't always exist. As an example, let's consider Category \textbf{2}: 
            \begin{center}
                \begin{pspicture}[showgrid=false](-1,2)(3.5,3.5)
                    \rput(3,3){\rnode{B}{B}}
                    \rput(0,3){\rnode{A}{A}}
                    \ncline[nodesep=4pt]{->}{A}{B}\naput{$f$}
                \end{pspicture}
            \end{center}
            It is easy to see that $ A $ and $ B $ have no product. However, if the $ A $ and $ B $ objects in a category have a product, then it is unique except for isomorphisms.
        \item
            The triple $ (P, p_1, p_2) $ will represent the product of the objects $ A $ and $ B $.
    \end{itemize}
\end{observacion}

As usual, let's first present a classic product example.

\begin{ejemplo}
   Consider the category \textbf{Set} and $A,B\in\textit{\textbf{Ob}}(\textbf{Set})$ objects. Recall that the usual Cartesian product of $ A $ and $ B $ is given by $A\times B = \{(a,b): a\in A \;\text{and}\; b\in B\}$ and the projections functions are defined by 
   \[
        \funcion{p_1}{A\times B}{A}{(a,b)}{a} \quad\text{y}\quad \funcion{p_2}{A\times B}{B}{(a,b)}{b.}
   \]
    The triple $(A\times B,p_1,p_2)$ is a product in \textbf{Set} of the objects $A$ y $B$. In fact, let $C\in\textit{\textbf{Ob}}(\mathcal{C})$ and $\func{f}{C}{A}$, $\func{f}{C}{B}$ arrows in \textbf{Set}, let's define the arrow $\overline{f}$ by
    \[
        \funcion{\overline{f}}{C}{A\times B}{c}{(f(c),g(c)).}
    \]
    
    Thus, it is easy to see that the following diagram
    \newpage
    
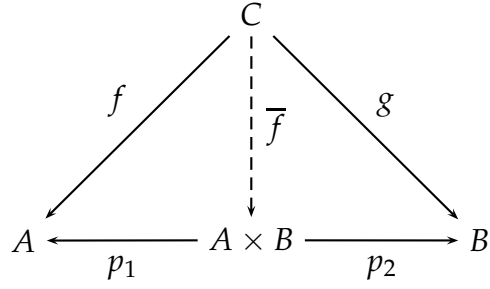
\begin{figure}[h]
    \centering
        \begin{pspicture}[showgrid=false](0,0)(6,4.5)
            \rput(0,1){\rnode{A}{$A$}}
            \rput(3,1){\rnode{B}{$A\times B$}}
            \rput(6,1){\rnode{C}{$B$}}
            \rput(3,4){\rnode{D}{$C$}}
        \ncline[nodesep=4pt]{<-}{A}{B}\nbput{$p_1$}
        \ncline[nodesep=4pt]{<-}{C}{B}\naput{$p_2$}
        \ncline[nodesep=4pt]{<-}{A}{D}\naput{$f$}
        \ncline[nodesep=4pt]{<-}{C}{D}\nbput{$g$}
        \ncline[nodesep=4pt,linestyle=dashed]{<-}{B}{D}\nbput{$\overline{f}$}
        \end{pspicture}
    \end{figure}
    conmutes.
\end{ejemplo}

Thanks to the preceding example we can consider the following product in the probabilistic context:

\begin{ejemplo}
  Consider the category \textbf{Meas} and $(X,\mathcal{F}_X),(Y,\mathcal{F}_Y)\in\textit{\textbf{Ob}}(\textbf{Meas})$. Also, let $(X\times X,\mathcal{F}_X\otimes\mathcal{F}_Y)$ be the measurable product space, where $\mathcal{F}_X\otimes\mathcal{F}_Y$ is the  smaller $\sigma-$algebra such that the projections functions $p_1$ and $p_2$ are measurable. On the other hand, let $(Z,\mathcal{F}_Z)\in\textit{\textbf{Ob}}(\textbf{Meas})$ and $\func{f}{(Z,\mathcal{F}_Z)}{(X,\mathcal{F}_X)}$, $\func{g}{(Z,\mathcal{F}_Z)}{(X,\mathcal{F}_X)}$ arrows in \textbf{Meas}, let's define the function $\overline{f}$ by
   \[
        \funcion{\overline{f}}{(Z,\mathcal{F}_Z)}{(X\times Y,\mathcal{F}_X\otimes \mathcal{F}_Y)}{c}{(f(c),g(c)).}
   \]
   It is easy to see that the following diagram

    
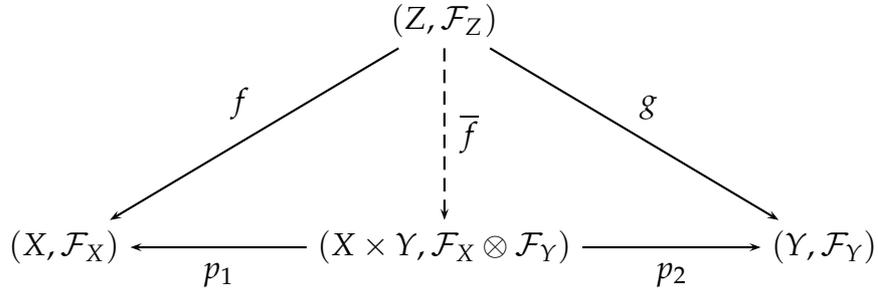
\begin{figure}[h]
        \centering
        \begin{pspicture}[showgrid=false](-2,0)(8,4)
            \rput(-2,0.5){\rnode{A}{$(X,\mathcal{F}_X)$}}
            \rput(3,0.5){\rnode{B}{$(X\times Y,\mathcal{F}_X\otimes\mathcal{F}_Y)$}}
            \rput(8,0.5){\rnode{C}{$(Y,\mathcal{F}_Y)$}}
            \rput(3,3.5){\rnode{D}{$(Z,\mathcal{F}_Z)$}}
        \ncline[nodesep=4pt]{<-}{A}{B}\nbput{$p_1$}
        \ncline[nodesep=4pt]{<-}{C}{B}\naput{$p_2$}
        \ncline[nodesep=4pt]{<-}{A}{D}\naput{$f$}
        \ncline[nodesep=4pt]{<-}{C}{D}\nbput{$g$}
        \ncline[nodesep=4pt,linestyle=dashed]{<-}{B}{D}\nbput{$\overline{f}$}
        \end{pspicture}
    \end{figure}
   
   conmutes.
   
   To prove that  $((X\times Y,\mathcal{F}_X\otimes\mathcal{F}_Y),p_1,p_2)$  is a product in \textbf{Meas},  just check that $ \func{\overline{f}}{(Z,\mathcal{F}_Z)}{(X\times Y,\mathcal{F}_X\otimes \mathcal{F}_Y)}$  is a meausurable $\mathcal{F}_Z/ \mathcal{F}_X\otimes\mathcal{F}_Y$ function. For this, let $F_X\times F_Y\in\mathcal{F}_X\otimes \mathcal{F}_Y$, we are going to show that  $\overline{f}^{-1}(F_X\times F_Y)\in\mathcal{F}_Z$. 
   
  Let $x\in\overline{f}^{-1}(F_X\times F_Y)$, we have the following equivalences
  \begin{align*}
       \overline{f}(x) \in F_X\times F_Y
       & \Longleftrightarrow (f(x),g(x)) \in F_X\times F_Y \\
       & \Longleftrightarrow f(x)\in F_X \;\text{ y }\; g(x)\in F_Y \\
       & \Longleftrightarrow x\in f^{-1} (F_X)\; \text{ y }\; x\in g^{-1}(F_Y),
   \end{align*}
   with which, we can conclude that
   \begin{equation}\label{eqn:ejem13}
        \overline{f}^{-1}(F_X\times F_Y) = f^{-1}(F_X)\cap g^{-1}(F_Y). 
   \end{equation}
   Now, given that $F_X\in\mathcal{F}_X$ and $F_Y\in\mathcal{F}_Y$, it follows that
   \[ 
        f^{-1}(F_X)\in \mathcal{F}_Z\quad\text{y}\quad g^{-1}(F_Y)\in\mathcal{F}_Z, 
   \]
   from where, together with \eqref{eqn:ejem13}, we can deduce that 
   \[
        \overline{f}^{-1}(F_X\times F_Y)\in\mathcal{F}_Z. 
   \]
   
   Therefore, $(X\times Y,\mathcal{F}_X\otimes\mathcal{F}_Y,p_1,p_2)$ is a product in \textbf{Meas}.
\end{ejemplo}

\section{Monads}

The most central concept in categorical probability is the probability monad, the first probability monad is the Giry monad defined in ~\cite{GiryM} and the first ideas about its structure are found in Lawvere~\cite{LawvereF}.

Before defining the concept of a monad, let's fix a certain notation. Let  $\mathcal{C},\mathcal{D},\mathcal{E}$ categories and  $\func{F,G}{\mathcal{C}}{\mathcal{D}}, \func{H}{\mathcal{D}}{\mathcal{E}}$ functors. If $\func{\eta}{F}{G}$ is a natural transformation, then the natural transformation $\func{H\eta}{H\circ F}{H\circ G }$ can be defined by
\[
    (H\eta)_A = H(\eta_A),  
\]
for each $A\in\textit{\textbf{Ob}}(\mathcal{A})$. 

On the other hand, if $\func{K}{\mathcal{E}}{\mathcal{D}}$ is a functor we can define the natural transformation $\func{\eta K}{F\circ K}{G\circ K}$ by  
\[
    (\eta K)_A = \eta_{K(A)}, 
\]

for each $A\in\textit{\textbf{Ob}}(\mathcal{E})$. 

\begin{definicion}[Monad]
    Let $\mathcal{C}$ a category. A \emph{monad} in $ \mathcal{C} $ consists of:
    \begin{itemize}
        \item 
            A functor $\func{T}{\mathcal{C}}{\mathcal{C}}$;
        \item
            A natural transformation $\func{\eta}{1_{\mathcal{C}}}{T}$, called unity; 
        \item
             A natural transformation $\func{\mu}{T\circ T}{T}$, called composition or multiplication; 
    \end{itemize}
    such that the following diagrams conmutes: 
    
    \begin{figure}[h]
    \centering
        \begin{pspicture}[showgrid=false](0,0)(13,4.5)
            \rput(0,4){\rnode{A}{$T$}}
            \rput(3,4){\rnode{B}{$T\circ T$}}
            \rput(3,1){\rnode{C}{$T$}}
            \rput(5,4){\rnode{D}{$T$}}
            \rput(8,4){\rnode{E}{$T\circ T$}}
            \rput(8,1){\rnode{F}{$T$}}
            \rput(10,4){\rnode{G}{$T\circ T\circ T$}}
            \rput(13,4){\rnode{H}{$T\circ T$}}
            \rput(13,1){\rnode{I}{$T$}}
            \rput(10,1){\rnode{J}{$T\circ T$}}
            \ncline[nodesep=4pt]{->}{A}{B}\naput{$\eta T$}
            \ncline[nodesep=4pt]{->}{B}{C}\naput{$\mu$}
            \ncline[nodesep=4pt]{->}{A}{C}\nbput{$\textbf{1}_{T}$}
            \ncline[nodesep=4pt]{->}{D}{E}\naput{$T\eta$}
            \ncline[nodesep=4pt]{->}{E}{F}\naput{$\mu$}
            \ncline[nodesep=4pt]{->}{D}{F}\nbput{$\textbf{1}_{T}$}
            \ncline[nodesep=4pt]{->}{G}{H}\naput{$T\mu$}
            \ncline[nodesep=4pt]{->}{H}{I}\naput{$\mu$}
            \ncline[nodesep=4pt]{->}{G}{J}\naput{$\mu T$}
            \ncline[nodesep=4pt]{->}{J}{I}\nbput{$\mu$}
        \end{pspicture}
    \label{fig:monada}
\end{figure}

These diagrams are called the left unit, right unit, and associativity, respectively.
\end{definicion}

\begin{observacion}
    We can rewrite the diagram in terms of the components of the natural transformations. To do this, let $A\in\textit{\textbf{Ob}}(\mathcal{C})$, we have the unit $\func{\eta_A}{A}{T(A)}$ and the composition $\func{\mu_A}{T(T(A))}{T(A)}$ are arrows in $\mathcal{C}$ such that the following diagrams conmutes: 
        \begin{figure}[h]
    \centering
        \begin{pspicture}[showgrid=false](0,0.5)(13,4.5)
            \rput(0,4){\rnode{A}{$T(A)$}}
            \rput(3,4){\rnode{B}{$T\circ T(A)$}}
            \rput(3,1){\rnode{C}{$T(A)$}}
            \rput(5,4){\rnode{D}{$T(A)$}}
            \rput(8,4){\rnode{E}{$T\circ T(A)$}}
            \rput(8,1){\rnode{F}{$T(A)$}}
            \rput(11,4){\rnode{G}{$T\circ T\circ T(A)$}}
            \rput(14,4){\rnode{H}{$T\circ T(A)$}}
            \rput(14,1){\rnode{I}{$T(A)$}}
            \rput(11,1){\rnode{J}{$T\circ T(A)$}}
            \ncline[nodesep=4pt]{->}{A}{B}\naput{$\eta_{T(A)}$}
            \ncline[nodesep=4pt]{->}{B}{C}\naput{$\mu_A$}
            \ncline[nodesep=4pt]{->}{A}{C}\nbput{$\textbf{1}_{T(A)}$}
            \ncline[nodesep=4pt]{->}{D}{E}\naput{$T(\eta_A)$}
            \ncline[nodesep=4pt]{->}{E}{F}\naput{$\mu_A$}
            \ncline[nodesep=4pt]{->}{D}{F}\nbput{$\textbf{1}_{T(A)}$}
            \ncline[nodesep=4pt]{->}{G}{H}\naput{$T (\mu_A) $}
            \ncline[nodesep=4pt]{->}{H}{I}\naput{$\mu_A$}
            \ncline[nodesep=4pt]{->}{G}{J}\naput{$\mu_{T(A)}$}
            \ncline[nodesep=4pt]{->}{J}{I}\nbput{$\mu_A$}
        \end{pspicture}
    \label{fig:monadacomp}
\end{figure}
\end{observacion}

 We will present an example of a monad in the category \textbf{Meas}.
 
 \begin{definicion}[Monad of Giry]\label{ejem:MonGiry}
      The monad of Giry is given by: 
  \begin{itemize}
      \item 
        The functor of Giry $\func{\mathcal{P}}{\textbf{Meas}}{\textbf{Meas}}$ defined in~\ref{defi:FGiry}; 
      \item
        \textbf{Unity}: The natural transformation $\func{\alpha}{1_{\textbf{Meas}}}{\mathcal{P}}$ presented in example~\ref{ejem:deltaDirac}; 
    \item
        \textbf{Composition:} 
        The natural transformation $\func{E}{\mathcal{P}\circ\mathcal{P}}{\mathcal{P}}$ defined by:

        Let $(X,\mathcal{F}_X)$a measurable space, for $\pi'\in \textbf{P}(\textbf{P}(X))$, define a measure on $(X,\mathcal{F}_X)$ by
\[
    E_{(X,\mathcal{F}_X)} (\pi') (F_X) = \int_{\mathbf{P}(X)} i_{F_X}(\mathbb{P}_X) \, \pi'(d \mathbb{P}_X) = \int_{\mathbf{P}(X)} \mathbb{P}_X(F_X) \, \pi'(d \mathbb{P}_X), 
\]
for each $F\in\mathcal{F}_X$. 
  \end{itemize}
 \end{definicion}

\begin{observacion}
    Note that $\mathbb{P}_X$ from the definition~\ref{ejem:MonGiry} corresponds to a probability measure defined on  $(X,\mathcal{F}_X)$. Furthermore, the integral is well defined because the evaluation functions are $\textbf{P}(\mathcal{F}_X)/\mathcal{B}([0,1])-$measurable. 
\end{observacion}

\begin{observacion}
    Let $(X,\mathcal{F}_X)$ and $(Y,\mathcal{F}_Y)$ be measurable spaces.
Let's prove that  $\func{E}{\mathcal{P}\circ\mathcal{P}}{\mathcal{P}}$ is a natural transformation, for which, we must show that  $\func{E_{(X,\mathcal{F}_X)}}{(\mathcal{P}\circ\mathcal{P}) (X,\mathcal{F}_X)}{\mathcal{P}(X,\mathcal{F}_X)}$ is a measurable function and that, for any arrow $\func{f}{(X,\mathcal{F}_X)}{(Y,\mathcal{F}_Y)}$ in \textbf{Meas}, the following diagram:

\begin{figure}[h]
    \centering
    \psset{unit=1.4cm}
            \begin{pspicture}[showgrid=false](0,0)(3,4.5)
                \rput(0,4){\rnode{A}{$\mathcal{P}(\mathcal{P}( (X,\mathcal{F_X})) )$}}
                \rput(3,4){\rnode{B}{$\mathcal{P}(\mathcal{P}( (Y,\mathcal{F_Y})) )$}}
                \rput(0,1){\rnode{C}{$\mathcal{P}((X,\mathcal{F}_X))$}}
                \rput(3,1){\rnode{D}{$\mathcal{P}((F,\mathcal{F}_Y))$}}
            \ncline[nodesep=4pt]{->}{A}{B}\naput{$\mathcal{P}(\mathcal{P}(f))$}
            \ncline[nodesep=4pt]{->}{C}{D}\nbput{$\mathcal{P}(f)$}
            \ncline[nodesep=4pt]{->}{A}{C}\nbput{$E_{(X,\mathcal{F}_X)}$}
            \ncline[nodesep=4pt]{->}{B}{D}\naput{$E_{(Y,\mathcal{F}_Y)}$}
            \end{pspicture}
    \end{figure}

commutes, which is equivalent to showing that
\begin{equation}\label{eqn:monadameas}
    E_{(Y,\mathcal{F}_Y)} \circ \mathcal{P}(\mathcal{P}(f)) = \mathcal{P}(f)\circ E_{(X,\mathcal{F}_X)}. 
\end{equation}

The fact that $\func{E_{(X,\mathcal{F}_X)}}{(\textbf{P}(\textbf{P}(X)),\textbf{P}(\textbf{P}(\mathcal{F}_X)))}{(\textbf{P}(X),\textbf{P}(\mathcal{F}_X))}$ is $\textbf{P}(\textbf{P}(X))/\textbf{P}(X)-$ measurable is analogous to the proof of the measurability of the natural transformation $\func{\alpha}{1_{\textbf{Meas}}}{\mathcal{P}}$ presented in example~\ref{ejem:deltaDirac}. Note that $\textbf{P}(\mathcal{F}_X)$ and $\textbf{P}(\textbf{P}(\mathcal{F}_X))$ correspond to the generated  $\sigma-$álgebras by the functions defined in en~\ref{defi:FGiry}.

Let's prove that the naturality condition \eqref{eqn:monadameas} is satisfied. Indeed, let $\pi'\in\textbf{P}(\textbf{P}(X))$ and $F_Y\in\mathcal{F}_Y$, on the one hand, thanks to the change of variable theorem, we have to
\begin{align}\label{eqn:transE_01}
     \left(E_{(Y,F_Y)} \circ \mathcal{P}(\mathcal{P}(f))(\pi')\right)(F_Y)
     & = \int_{\textbf{P}(Y)}  i_{F_Y} \; d((\mathcal{P}(\mathcal{P})(f))(\pi') ) \notag \\
     & = \int_{\textbf{P}(Y)}  i_{F_Y} \; d( \pi' \circ (\mathcal{P}(f))^{-1} ) \notag\\
     & = \int_{\mathcal{P}(f)^{-1}(\textbf{P}(Y))}  i_{F_Y}\circ \mathcal{P}(f) \; d\pi' \notag\\
     & = \int_{\textbf{P}(X)} i_{F_Y}\circ \mathcal{P}(f) \; d\pi',\notag \\
     & = \int_{\textbf{P}(X)} i_{F_Y}\circ \mathcal{P}(f)(\mathbb{P}_X) \, \pi'(d\mathbb{P}_X), \notag\\
     & = \int_{\textbf{P}(X)} i_{F_Y}\left(\mathbb{P}_X\circ f^{-1}\right) \; \pi'(d\mathbb{P}_X),\notag \\
     & = \int_{\textbf{P}(X)} \mathbb{P}_X\left(f^{-1}(F_Y) \right)\; \pi'(d\mathbb{P}_X).
\end{align}

On the other hand, we have to
\begin{align}\label{eqn:transE_02}
    (\mathcal{P}(f)\circ E_{(X,\mathcal{F}_X)}(\pi')) (F_Y)
    & = E_{(X,\mathcal{F}_X)}  (\pi') (f^{-1}(F_Y)) \notag\\
    & = \int_{\textbf{P}(X)} i_{f^{-1} (F_Y)} \, d\pi' \notag\\
    & = \int_{\textbf{P}(X)} i_{f^{-1} (F_Y)}(\mathbb{P}_X) \, \pi'(d\mathbb{P}_X) \notag\\
    & = \int_{\textbf{P}(X)} \mathbb{P}_X \left( f^{-1}(F_Y)\right) \, \pi'(d\mathbb{P}_X).
\end{align}

Together with  \eqref{eqn:transE_01} and \eqref{eqn:transE_02}, we can conclude that
\[
     E_{Y,\mathcal{F}_Y} \circ \mathcal{P}(\mathcal{P}(f)) = \mathcal{P}(f)\circ E_{(X,\mathcal{F}_X)}.  
\]

\end{observacion}

Consider the theorem presented in Giry [theorem 1, \cite{GiryM}], which allows us to show that the unit and associativity diagrams commute. 

\begin{lema}\label{lema:monadaGirylema}
    Let $(X,\mathcal{F}_X)$, $(Y,\mathcal{F}_Y)$ measurables spaces, $\mathbb{P}_X\in\textbf{P}(X)$, $\pi'\in \textbf{P}(\textbf{P}(X))$ and $\func{\theta}{(X,\mathcal{F}_X)}{(\R,\mathcal{B}(\R))}$ an integrable function. The following results are obtained:
    \begin{enumerate}
        \item 
            $\int_X \theta  \, d\alpha_{(X,\mathcal{F}_X)}(x) = \theta(x)$ for all $x\in X$. 
        \item
            The function $\func{\xi_{\theta}}{(\textbf{P}(X),\textbf{P}(\mathcal{F}_X))}{(\R,\mathcal{B}(\R))}$ defined by $\xi_{\theta}(\mathbb{P}_X) = \displaystyle\int_X \theta \, d\mathbb{P}_X$ is $\textbf{P}(\mathcal{F}_X)/\mathcal{B}(\R)-$ measurable.
        \item
            $\displaystyle\int_X \theta(x) \, d (E_{(X,\mathcal{F}_X)} (\pi')) = \displaystyle\int_{\textbf{P}(X)} \xi_{\theta}(\mathbb{P}_X)\, \pi'(d\mathbb{P}_X)$. 
    \end{enumerate}
\end{lema}

\begin{proof}
    \begin{enumerate}
        \item 
            Let $x\in X$, we have that
            \begin{equation*}
                \int_X \theta  \, d\alpha_{(X,\mathcal{F}_X)}(x) = \int_X \theta  \, d\delta_x,
            \end{equation*}
            whence, since $ \delta_x $ is the Dirac measure at the point $ x \in X $, it follows that
            \[
                 \int_X \theta  \, d\alpha_{(X,\mathcal{F}_X)}(x) = \theta(x).
            \]
        \item
            First, we will prove propositions 2 and 3 for indicator functions, let $F_X\in\mathcal{F}_X$. 
            \begin{itemize}
                \item 
                   Let's consider $\theta = \mymathbb{1}_{F_Z}$, on the one hand, we have to 
                    \begin{align}\label{eqn:medxi}
                        \xi_{\theta}(\mathbb{P}_X)
                        & = \int_X \theta \, d\mathbb{P}_X, \notag\\
                        & = \int_X \mymathbb{1}_{F_X} \, d\mathbb{P}_X,\notag \\
                        & = \int_{F_X} \, d\mathbb{P}_X, \notag\\
                        & = \mathbb{P}_X(F_X). 
                    \end{align}

                    Let $Y\in\tau$, we are going to show that $\xi_{\theta}$ function $\textbf{P}(\mathcal{F}_X)/\mathcal{B}(\R)-$measurable, which is equivalent to showing that $\xi_{\theta}^{-1}(Y) \in \sigma(\{i_{F_x}^{-1}(B): F_X\in\mathcal{F}_X \ \text{ y }\, B\in\mathcal{B}(\R) \})$. Thanks to \eqref{eqn:medxi}, let's note that 
                    \begin{align*}
                        \xi_{\theta}^{-1}(Y)
                        & = \{\mathbb{P}_X\in\textbf{P}(X): \xi_{\theta}(\mathbb{P}_X)\in Y \}, \\
                        & = \{\mathbb{P}_X\in\textbf{P}(X): \mathbb{P}_X(F_X)\in Y \}, \\
                        & = i_{F_X}^{-1}(Y), 
                    \end{align*}
                    with which given that $F_X\in\mathcal{F}_X$ and $Y\in\mathcal{B}(\R)$, it follows that 
                    \[
                        \xi_{\theta}^{-1}(Y) \in \sigma(\{i_{F_x}^{-1}(B): F_X\in\mathcal{F}_X \ \text{ and }\, B\in\mathcal{B}(\R) \}),
                    \]
                    that is, we have proven that $\xi_{\theta}$ es a measurable  $\textbf{P}(\mathcal{F}_X)/\mathcal{B}(\R)$ function.
                    
                \item
                   Similarly, let $\theta=\mymathbb{1}_{F_X}$, thanks to the preceding result, we have that $\xi_{\theta}(\mathbb{P}_X) = \mathbb{P}_X(F_X)$ for all $\mathbb{P}_X\in\textbf{P}(X)$, with which, it follows that  
                    \begin{equation}\label{eqn:lema2.2_01}
                        \int_{\textbf{P}(X)} \xi_{\theta}(\mathbb{P}_X) \, \pi'(d\mathbb{P}_X) = \int_{\textbf{P}(X)} \mathbb{P}_X(F_X) \, \pi'(d\mathbb{P}_X).
                    \end{equation}
                    
                   Now, we have to
                    \begin{align*}
                        \int_X \theta \, d E_{(X,\mathcal{F}_X)} (\pi')
                        & = \int_X \mymathbb{1}_{F_X} \, d E_{(X,\mathcal{F}_X)} (\pi'), \\
                        & = \int_{F_X}  \, d E_{(X,\mathcal{F}_X)} (\pi'), \\
                        & = E_{(X,\mathcal{F}_X)} (\pi') (F_X), \\
                        & = \int_{\textbf{P}(X)} \mathbb{P}_X(F_X) \ \pi'(d\mathbb{P}_X),
                    \end{align*}
                    from where, together with \eqref{eqn:lema2.2_01}, we can conclude that 
                    \[
                        \displaystyle\int_X \theta \, d E_{(X,\mathcal{F}_X)} (\pi') = \displaystyle\int_{\textbf{P}(X)} \xi_{\theta}\, \pi'(\mathbb{P}_X).
                    \]
            \end{itemize}
            
           By the linearity of the integral, the results of \textit{2.} and \textit{3.} hold when $ \theta $ is a simple function. The general case is followed by the monotone convergence theorem and the fact that $ \theta $ can be seen as the limit of an increasing sequence of simple functions. 
    \end{enumerate}
\end{proof}

\begin{observacion}

The preceding lemma allows us to prove that the following diagrams:

 \begin{figure}[h]
    \centering
        \begin{pspicture}[showgrid=false](0,0)(13,4.5)
            \rput(0,4){\rnode{A}{$\mathcal{P}$}}
            \rput(3,4){\rnode{B}{$\mathcal{P}\circ\mathcal{P}$}}
            \rput(3,1){\rnode{C}{$\mathcal{P}$}}
            \rput(5,4){\rnode{D}{$\mathcal{P}$}}
            \rput(8,4){\rnode{E}{$\mathcal{P}\circ \mathcal{P}$}}
            \rput(8,1){\rnode{F}{$\mathcal{P}$}}
            \rput(10,4){\rnode{G}{$\mathcal{P}\circ \mathcal{P}\circ \mathcal{P}$}}
            \rput(13,4){\rnode{H}{$\mathcal{P}\circ\mathcal{P}$}}
            \rput(13,1){\rnode{I}{$\mathcal{P}$}}
            \rput(10,1){\rnode{J}{$\mathcal{P}\circ \mathcal{P}$}}
            \ncline[nodesep=4pt]{->}{A}{B}\naput{$\alpha\circ\mathcal{P}$}
            \ncline[nodesep=4pt]{->}{B}{C}\naput{$E$}
            \ncline[nodesep=4pt]{->}{A}{C}\nbput{$\textbf{1}_{\mathcal{P}}$}
            \ncline[nodesep=4pt]{->}{D}{E}\naput{$\mathcal{P}\circ\alpha$}
            \ncline[nodesep=4pt]{->}{E}{F}\naput{$E$}
            \ncline[nodesep=4pt]{->}{D}{F}\nbput{$\textbf{1}_{\mathcal{P}}$}
            \ncline[nodesep=4pt]{->}{G}{H}\naput{$\mathcal{P}\circ E$}
            \ncline[nodesep=4pt]{->}{H}{I}\naput{$E$}
            \ncline[nodesep=4pt]{->}{G}{J}\naput{$E\circ\mathcal{P}$}
            \ncline[nodesep=4pt]{->}{J}{I}\nbput{$E$}
        \end{pspicture}
\end{figure}

commute. The natural transformations  $\eta\circ \mathcal{P},\mathcal{P}\circ\eta$ y $E\circ\mathcal{P}$, are defined by  
\[
    (\alpha\circ\mathcal{P})_{(X,\mathcal{F}_X)} = \alpha_{\mathcal{P}(X,\mathcal{F}_X)}, (\mathcal{P}\circ\alpha)_{(X,\mathcal{F}_X)} = \mathcal{P}(\alpha_{(X,\mathcal{F}_X)} ) \; \text{ y }\; (E\circ \mathcal{P})_{(X,\mathcal{F}_X)} = E_{\mathcal{P}(X,\mathcal{F}_X)}, 
\]
for each $(X,\mathcal{F}_X)\in\textit{\textbf{Ob}}(\textbf{Meas})$, respectively.

Let's prove that the left unit, right unit, and associativity diagrams switch. Let $ (X, \mathcal{F}_X) $ be a measurable space.

\begin{enumerate}
    \item 
      Thanks to proposition 1 of the lemma~\ref{lema:monadaGirylema}, it is easy to see that
        \[
            E_{(X,\mathcal{F}_X)}\circ \delta_{\mathcal{P}(X,\mathcal{F}_X)} = \mymathbb{1}_{(X,\mathcal{F}_X)} \quad\text{y}\quad E_{(X,\mathcal{F}_X)}\circ \mathcal{P}(\delta_{(X,\mathcal{F}_X)}) = \mymathbb{1}_{(X,\mathcal{F}_X)}. 
        \]
    \item
        Now we are going to prove that the associativity diagram commutes, for this, let  $\pi''\in \textbf{P}(\textbf{P}(\textbf{P}(X)))$ and $F_X\in\mathcal{F}_X$, we must show that 
        \[
            E_{(X,\mathcal{F}_X)} \circ E_{\mathcal{P}(X,\mathcal{F}_X)}(\pi'')(F_X) = E_{(X,\mathcal{F}_X)}\circ\alpha_{(\mathcal{P}{(X,\mathcal{F}_X))}}(\pi'')(F_X).
        \]

      On the one hand, together with the variable change theorem, we obtain that 
         \begin{align}\label{eqn:asociatividadGiry_01}
            (E_{(X,\mathcal{F}_X)} \circ \mathcal{P}(E_{\mathcal{P}(X,\mathcal{F}_X)})(\pi''))(F_X)
            & = \int_{\textbf{P}(X)} i_{F_X} \, d\left( \mathcal{P}( E_{(X,F_X)}) (\pi'') \right) \notag \\
            & = \int_{\textbf{P}(X)} i_{F_X} \, d (\pi''\circ E_{\mathcal{P}(X,\mathcal{F}_X)}^{-1}) , \notag \\
            & = \int_{E_{\mathcal{P}(X,\mathcal{F}_X)}^{-1}(\textbf{P}(X))} i_{F_X} \circ E_{(X,\mathcal{F}_X)} \, d\pi'', \notag \\
            & =  \int_{ \textbf{P}(\textbf{P}(X))} i_{F_X} \circ E_{(X,\mathcal{F}_X)} (\pi') \, \pi''(d\pi'), \notag \\
            & = \int_{ \textbf{P}(\textbf{P}(X))}  (E_{(X,\mathcal{F}_X)} (\pi'))(F_X) \, \pi''(d\pi'). 
        \end{align}
        
        Now, by the definition of $E_{(X,\mathcal{F}_X)}$, we have 
        \[
           (E_{(X,\mathcal{F}_X)} (\pi'))(F_X) = \int_{\textbf{P}(X)} i_{F_X} \, d\pi' = \xi_{i_{F_x}}(\pi'), 
        \]
        from where, together with  \eqref{eqn:asociatividadGiry_01} and by proposition 2 of the lemma~\ref{lema:monadaGirylema}, we obtain that  
        
        \[
            (E_{(X,\mathcal{F}_X)} \circ \mathcal{P}(E_{\mathcal{P}(X,\mathcal{F}_X)})(\pi''))(F_X) = \int_{ \textbf{P}(\textbf{P}(X))}  \xi_{i_{F_X}} (\pi') \, \pi''(d\pi'), 
        \]
        with which, thanks to proposition 3 of the lemma~\ref{lema:monadaGirylema}, it follows that 
        \begin{align}\label{eqn:asociatividadGiry_03}
             (E_{(X,\mathcal{F}_X)} \circ \mathcal{P}(E_{\mathcal{P}(X,\mathcal{F}_X)})(\pi''))(F_X)
             & = \int_{\textbf{P}(X)} i_{F_X} \, d(E_{\mathcal{P}(X,\mathcal{F}_X)}(\pi')), \notag \\
             & = \int_{\textbf{P}(X)} i_{F_X} (\mathbb{P}_X) \, (E_{\mathcal{P}(X,\mathcal{F}_X)}(\pi'))(d\mathbb{P}_X). 
        \end{align}

       On the other hand, by definition of $E$, we have
        \begin{equation}\label{eqn:asociatividadGiry_02}
            (E_{(X,\mathcal{F}_X)} \circ E_{\mathcal{P}(X,\mathcal{F}_X)}(\pi''))(F_X) = \int_{\textbf{P}(X)} i_{F_X}(\mathbb{P}_X) \,  (E_{(X,\mathcal{F}_X)}(\pi'')) d(\mathbb{P}_X).
        \end{equation}
        
        Thus, together with \eqref{eqn:asociatividadGiry_03} y \eqref{eqn:asociatividadGiry_02}, we can conclude that 
        \[
            E_{(X,\mathcal{F}_X)} \circ E_{\mathcal{P}(X,\mathcal{F}_X)}(\pi'')(F_X) = E_{(X,\mathcal{F}_X)}\circ\mathcal{P}(E_{(X,\mathcal{F}_X)})(\pi'')(F_X). 
        \]
\end{enumerate}

\end{observacion}

\nocite{*}

\end{document}